\documentclass[a4paper, 10pt, twoside, notitlepage]{amsart}

\usepackage{amsmath,amscd}
\usepackage{amssymb}
\usepackage{amsthm}
\usepackage{comment}
\usepackage{graphicx, xcolor}

\usepackage{mathrsfs}
\usepackage[ocgcolorlinks, linkcolor=blue]{hyperref}

\usepackage{bm}
\usepackage{bbm}
\usepackage{url}

\usepackage[utf8]{inputenc}
\usepackage{mathtools,amssymb}
\usepackage{esint}
\usepackage{tikz}
\usepackage{dsfont}
\usepackage{relsize}
\usepackage{url}
\usepackage{xcolor}
\usepackage{graphicx}
\usepackage{mathrsfs}
\usepackage[shortlabels]{enumitem}
\usepackage{lineno}
\usepackage{amsmath}
\usepackage{enumitem}
\usepackage{amsthm} 
\usepackage{verbatim}
\usepackage{dsfont}
\numberwithin{equation}{section}

\allowdisplaybreaks

\mathtoolsset{showonlyrefs}

\graphicspath{{images/}}

\newtheorem{theorem}{Theorem}[section]
\newtheorem{lemma}[theorem]{Lemma}
\newtheorem{definition}[theorem]{Definition}
\newtheorem{corollary}[theorem]{Corollary}

\newtheorem{claim}[theorem]{Claim}
\newtheorem{proposition}[theorem]{Proposition}
\newtheorem{question}{Question}
\newtheorem{remark}[theorem]{Remark}
\newtheorem{assumption}[theorem]{Assumption}

\title[Inverse problems for nonlinear nonlocal wave equations]{Well-posedness and inverse problems for semilinear nonlocal wave equations}

\author[Y.-H.~Lin]{Yi-Hsuan Lin}
\address{Department of Applied Mathematics, National Yang Ming Chiao Tung University, Hsinchu, Taiwan}
\email{yihsuanlin3@gmail.com}

\author[T.~Tyni]{Teemu Tyni}
\address{Unit of applied and computational mathematics, University of Oulu, Finland}
\email{teemu.tyni@oulu.fi}

\author[P. Zimmermann]{Philipp Zimmermann}
\address{Department of Mathematics, ETH Zurich, Z\"urich, Switzerland \& Departament de Matem\`atiques i Inform\`atica, Universitat de Barcelona, Barcelona, Spain}
\email{philipp.zimmermann@ub.edu}
\newcommand{\R}{{\mathbb R}}

\newcommand{\N}{{\mathbb N}}

\newcommand{\eps}{\varepsilon}

\newcommand {\p} {\partial}

\newcommand{\LC}{\left(}
\newcommand{\RC}{\right)}
\newcommand{\wt}{\widetilde}

\newcommand{\schwartz}{\mathscr{S}}

\newcommand{\tempered}{\mathscr{S}^{\prime}}

\newcommand{\fourier}{\mathcal{F}}
\newcommand{\ifourier}{\mathcal{F}^{-1}}

\newcommand{\distr}{\mathscr{D}^{\prime}}%distribution
\newcommand{\abs}[1]{\left\lvert #1 \right\rvert}%absolute value
\DeclareMathOperator{\supp}{supp} %support
\DeclareMathOperator{\loc}{loc} %distance

\newcommand{\weak}{\rightharpoonup}%weak convergence 
\newcommand{\weakstar}{\overset{\ast}{\rightharpoonup}}%weak star convergence 

\begin{document}

	\maketitle
	\begin{abstract}
		This article is devoted to forward and inverse problems associated with time-independent semilinear nonlocal wave equations. We first establish comprehensive well-posedness results for some semilinear nonlocal wave equations. The main challenge is due to the low regularity of the solutions of linear nonlocal wave equations. We then turn to an inverse problem of recovering the nonlinearity of the equation. More precisely, we show that the exterior Dirichlet-to-Neumann map uniquely determines homogeneous nonlinearities of the form $f(x,u)$ under certain growth conditions. On the other hand, we also prove that initial data can be determined by using passive measurements under certain nonlinearity conditions. The main tools used for the inverse problem are the unique continuation principle of the fractional Laplacian and a Runge approximation property. The results hold for any spatial dimension $n\in \N$.
		
		\medskip
		
		

	\end{abstract}

	\tableofcontents

	\section{Introduction}
	\label{sec: introduction}

	\subsection{Mathematical model and main results}
	In this paper, we study forward and inverse problems for \emph{semilinear nonlocal wave equations}. We prove that an a priori unknown nonlinearity $f(x,\tau)$, belonging to a certain subclass of Carath\'eodory functions, can be uniquely recovered from the \emph{Dirichlet-to-Neumann} (DN) map related to the semilinear nonlocal wave equation
	\begin{equation}
		\label{eq: nonlinear wave equation}
		\begin{cases}
			\partial_t^2u +(-\Delta)^su + f(x,u)=0 &\text{ in }\Omega_T,\\
			u=\varphi  &\text{ in }(\Omega_e)_T,\\
			u(0)=u_0, \quad \partial_t u(0)=u_1 &\text{ in }\Omega.
		\end{cases}
	\end{equation}
	Here and throughout this article $\Omega\subset\R^n$ is a bounded Lipschitz domain for $n\in \N$ with exterior $\Omega_e := \R^n\setminus \overline{\Omega}$, $T>0$ a finite time horizon and $s>0$ a non-integer. Moreover, we set $A_t\vcentcolon= A\times (0,t)$, for any subset $A\subseteq \R^n$ and $t>0$. In our study the magnitude of the time horizon does not play any role. 
	The nonlocal wave equation appears, for instance, as a special case in the study of \emph{peridynamics}, which is a nonlocal elasticity theory particularly used to study material dynamics with discontinuities, see e.g., \cite{peridynamics}.
	% e.g. (2.25) with scalar u and C(x,y)=|x-y|^{-n-2s}
	For the time being let us assume the well-posedness of \eqref{eq: nonlinear wave equation} in a suitable function space. Then, for any two measurement sets $W_1,W_2\subset\Omega_e$, let us formally introduce the  DN map $\Lambda^{f}_{u_0,u_1}$ of \eqref{eq: nonlinear wave equation} via
	\begin{equation}\label{DN map}
		\begin{split}
			\Lambda^{f}_{u_0,u_1}\varphi=\left. (-\Delta)^s u_\varphi \right|_{(W_2)_T},
		\end{split}
	\end{equation}
	for any $\varphi\in C_c^{\infty}((W_1)_T)$, where $u_\varphi \colon \R^n_T\to \R$ denotes the unique solution of \eqref{eq: nonlinear wave equation}. The DN map will be rigorously defined in Section~\ref{sec: DN}. In our study we want to answer:
	\begin{question}
		Can one determine under suitable assumptions the nonlinearity $f$ or the initial conditions $u_0,u_1$?
	\end{question}

	In the special case of $f$ being homogeneous and satisfying the conditions in Assumption~\ref{main assumptions on nonlinearities}, we are able to establish the following positive results.
	
	\begin{theorem}[Recovery of the nonlinearity]\label{Thm: recovery of nonlinearity}
		Let $\Omega \subset \R^n$ be a bounded domain with Lipschitz boundary, $T>0$ and $s>0$ a non-integer. Let $W_1,W_2\subset \Omega_e$ be open sets. Suppose the nonlinearities $f_j$ satisfy the conditions in Assumption \ref{main assumptions on nonlinearities} with $F(x,\tau)\geq 0$.
		Suppose also that $f_j(x,\tau)$ are $(r+1)$-homogeneous with $r$ as in Assumption~\ref{main assumptions on nonlinearities} and $0<r\leq 1$.
		Let $\Lambda_j \vcentcolon=\Lambda^{f_j}_{0,0}$ be the DN maps of
		\begin{equation}
			\label{eq: nonlinear wave equation j=1,2}
			\begin{cases}
				\partial_t^2u_j +(-\Delta)^su_j + f_j(x,u_j)=0 &\text{ in }\Omega_T,\\
				u_j=\varphi  &\text{ in }(\Omega_e)_T,\\
				u_j(0)=\p_t u_j(0)=0 &\text{ in }\Omega,
			\end{cases}
		\end{equation}
		for $j=1,2$, satisfying
		\begin{equation}\label{same DN map in thm}
			\Lambda_1(\varphi)\big|_{(W_2)_T}=\Lambda_2(\varphi)\big|_{(W_2)_T}  
		\end{equation}
		for any $\varphi \in C^\infty_c((W_1)_T)$. Then there holds $f_1(x,\tau)=f_2(x,\tau)$ for almost all $x\in\Omega$ and $\tau\in\R$.
	\end{theorem}
	
	Moreover, with the nonlocality at hand, we are able to determine the initial data without any knowledge of the nonlinearities.

	\begin{theorem}[Recovery of the initial values by passive measurements]\label{Thm: recovery of initial values}
		Let $\Omega \subset \R^n$ be a bounded domain with Lipschitz boundary, $T>0$ and $s>0$ a non-integer. Suppose the nonlinearities $f_j$ satisfy the conditions in Assumption \ref{main assumptions on nonlinearities} and $(u_{0,j},u_{1,j})\in \wt H^s(\Omega)\times L^2(\Omega)$. Let $\Lambda_j\vcentcolon=\Lambda^{f_j}_{u_{0,j},u_{1,j}}$ be the DN maps of 
		\begin{equation}\label{equation det of initial j=12}
			\begin{cases}
				\partial_t^2u_j +(-\Delta)^su_j + f_j(x,u_j) =0 &\text{ in }\Omega_T,\\
				u_j=\varphi  &\text{ in }(\Omega_e)_T,\\
				u_j(0)=u_{0,j}, \quad \partial_t u_j(0)=u_{1,j} &\text{ in }\Omega,
			\end{cases}
		\end{equation}
		for $j=1,2$. Suppose that 
		$\Lambda_1(0)\big|_{(W_2)_T}=\Lambda_2(0)\big|_{(W_2)_T}$ holds, then one has 
		\begin{equation}
			u_{0,1}=u_{0,2},\quad u_{1,1}=u_{1,2} \text{ in }\Omega.
		\end{equation}
	\end{theorem}

	\begin{remark}
		Let us emphasize the following observations.
		\begin{enumerate}[(i)]
			\item As the exterior data $\varphi =0$ in $(\Omega_e)_T$, this measurement is usually regarded as the \emph{passive measurement}\footnote{Here the passive measurements are generated by unknown sources, without injecting new inputs or affecting the existing one.}, which forms a \emph{single measurement} result. 
			
			\item The preceding theorem states that even though we might not be able to determine the unknown nonlinearity in general, we can still recover the initial data. Analogous results in the local case were studied in the work \cite{LLL2024determining}. However, in order to determine the initial data in the local case, one needs some extra tools, such as observability estimate for linear wave equations. Nevertheless, the nonlocality helps us to recover the initial condition without any observability estimates. 
		\end{enumerate}
		
	\end{remark}
	
	In addition, when the coefficient $f(x,u)=a(x)u$, for sufficiently regular coefficient $a(x)$, we are able to determine both initial data and $a(x)$ simultaneously.
	
	\begin{corollary}[Simultaneous recovery of both initial data and coefficients]\label{Cor: simul near}
		Let $\Omega \subset \R^n$ be a bounded domain with Lipschitz boundary, $T>0$ and $s>0$ a non-integer. Suppose either $0\leq a_j\in L^p(\Omega)$ for some $p$ satisfying \eqref{eq: cond on p}, or $a_j\in L^\infty(\Omega)$, and $(u_{0,j},u_{1,j})\in \wt H^s(\Omega)\times L^2(\Omega)$ for $j=1,2$. Let $\Lambda_j\vcentcolon=\Lambda^{a_j}_{u_{0,j},u_{1,j}}$ be the DN maps of the linear nonlocal wave equation 
		\begin{equation}\label{linear equation det of initial j=12}
			\begin{cases}
				\partial_t^2u_j +(-\Delta)^su_j + a_j(x)u_j =0 &\text{ in }\Omega_T,\\
				u_j=\varphi  &\text{ in }(\Omega_e)_T,\\
				u_j(0)=u_{0,j}, \quad \partial_t u_j(0)=u_{1,j} &\text{ in }\Omega,
			\end{cases}
		\end{equation}
		for $j=1,2$. Suppose that 
		\begin{equation}\label{same DN map in thm 2}
			\Lambda_1(\varphi)\big|_{(W_2)_T}=\Lambda_2(\varphi)\big|_{(W_2)_T}  
		\end{equation}
		for any $\varphi \in C^\infty_c((W_1)_T)$, then there holds
		\begin{equation}
			u_{0,1}=u_{0,2},\quad u_{1,1}=u_{1,2} \quad  \text{and}\quad a_1=a_2 \text{ in }\Omega.
		\end{equation}
		
	\end{corollary}

	\subsection{Earlier literature}
	Inverse problems for nonlocal partial differential equations (PDEs) have been widely studied in the recent years. In the pioneering work \cite{GSU20} the authors determined unknown potential $q$ in the fractional Schr\"odinger equation $(-\Delta)^s u +q u =0$ in a bounded open set by using its exterior DN map. The main tools in their proof are the \emph{unique continuation property} (UCP) for the fractional Laplacian, and the \emph{Runge approximation property}. These two remarkable ingredients help us to study several inverse problems that are widely open in the local setting. For example, both a surrounding potential and an unknown inclusion can be determined simultaneously as shown in \cite{CLL2017simultaneously}. If-and-only-if -monotonicity relations have been derived in \cite{harrach2017nonlocal-monotonicity,harrach2020monotonicity,lin2020monotonicity}, and a determination of both drift and potentials result was proved in \cite{cekic2020calderon}. The aforementioned results are either open, or not true for their local counterparts. In short, we regard the nonlocality as a tool in solving related inverse problems.

	So far, most of the existing works in this area are focused on the determination or reconstruction of lower order coefficients via their DN maps (for example, see \cite{RS17,GRSU18,CMRU20}). Meanwhile, some nonlinear nonlocal inverse problems have been addressed in \cite{LL2020inverse,LL2022inverse}. Very recently, the recovery of leading order coefficients in nonlocal operators have been investigated, such as \cite{CGRU2023reduction,LLU2023calder}, where the authors found a novel reduction formula from the nonlocal to the local case. We also refer readers to \cite{zimmermann2023inverse,RZ2022LowReg,CRTZ-2022} and references therein about inverse problems for different type of nonlocal operators. Furthermore, there are even uniqueness results for leading order coefficients in nonlinear nonlocal equations like for equations of porous medium type \cite{LZ2023unique} or of $p$-Laplace type \cite{KRZ-2023,KLZ-2022}.
	
	Inverse problems for nonlinear (local) hyperbolic equations have been widely investigated. In fact, it is now known that the nonlinear interaction of waves can generate new waves, which can be taken  advantage of when studying related inverse problems. This field started from \cite{KLU18}, where it was proved that the local measurements can determine global topology and differentiable structure uniquely for a semilinear wave equation with a quadratic nonlinearity.  In further, inverse problems were studied for general semilinear wave equations on Lorentzian manifolds \cite{LUW18}, and related inverse problems were investigated for the Einstein-Maxwell equation in \cite{LUW17arXiv}. We also refer to \cite{KLOU14,dHUW18,WZ19,LLL2024determining,LLPT20,LLPT21,LLPT24}, and the many fruitful references therein for more inverse problems for hyperbolic PDEs.
	
	Next, let us mention that there are also a few results on inverse problems for nonlocal wave equations. On the one hand in \cite{KLW2022} it has been shown that from the DN map related to 
 \[
\begin{cases}
			\partial_t^2u +(-\Delta)^su + qu=0 &\text{ in }\Omega_T,\\
			u=\varphi  &\text{ in }(\Omega_e)_T,\\
			u(0)=u_0, \quad \partial_t u(0)=u_1 &\text{ in }\Omega
		\end{cases}
 \]
 one can uniquely determine the potential $q\in L^{\infty}(\Omega)$, when $0<s<1$. On the other hand in \cite{zimmermann2024calderon}, the same unique determination question has been studied for the so called \emph{(nonlinear) nonlocal viscous wave equation}
	 \begin{equation}
  \label{eq: viscous wave eq}
\begin{cases}
			\partial_t^2u +(-\Delta)^s\partial_t u+(-\Delta)^su + f(u)=0 &\text{ in }\Omega_T,\\
			u=\varphi  &\text{ in }(\Omega_e)_T,\\
			u(0)=u_0, \quad \partial_t u(0)=u_1 &\text{ in }\Omega.
		\end{cases}
 \end{equation}
 In this article it has been established that unique determination of $f$ holds if
 \begin{enumerate}
     \item $f(u)=qu$ with $q\in L^{\infty}(0,T;L^p(\Omega))$, where $p$ satisfies \eqref{eq: cond on p}, and $q$ has a certain continuity property in the time variable

     \item or $f$ satisfying the condition (i) in Assumption~\ref{main assumptions on nonlinearities}, being $r+1$ homogeneous and $0<r\leq 2$.
 \end{enumerate}
 Let us point out that a main difference between the problems \eqref{eq: nonlinear wave equation} and \eqref{eq: viscous wave eq} is that in the later case one can establish a Runge approximation theorem in $L^2(0,T;\widetilde{H}^s(\Omega))$ (see \cite[Proposition~4.2]{zimmermann2024calderon}) instead of $L^2(\Omega_T)$. This in turn allows to handle larger $r$ values. Finally, let us recall that the difference of these approximation results rest on the fact that solutions to the equation \eqref{eq: viscous wave eq} are much more regular as for \eqref{eq: nonlinear wave equation} and in fact the latter can be obtained by a certain approximation process of the first one (see \cite{LionsMagenesVol1} or Claim~\ref{claim: integration by parts} below). Actually, the loss term $(-\Delta)^s\partial_t$ regularizes the solution $u$ to \eqref{eq: viscous wave eq} such that $\partial_t u$ belongs to $L^2(0,T;H^s(\R^n))$, whereas the solution $v$ of the nonlocal wave equation \eqref{eq: nonlinear wave equation} only satisfies $\partial_t v\in L^2(\R^n_T)$.
	
	\subsection{Organization of the paper}
	We start in Section~\ref{sec: preliminaries} by defining the used function spaces, the fractional Laplacian, and by recalling their main properties. Section~\ref{sec: well-posedness} concerns the well-posedness of the nonlinear nonlocal wave equation. Under certain decay assumptions we show the existence, uniqueness and energy estimates for the solutions. We also define the exterior DN map, which will be the measurement data used for inverse problems later on. Finally, in Section~\ref{sec: inverse problem} we study the continuity of the nonlinear terms, prove a Runge approximation theorem under arbitrary initial data, and finally establish determination of the nonlinear terms and initial data, given the DN map.

	\section{Preliminaries}
	\label{sec: preliminaries}
	
	Throughout this article the space dimension $n$ is a fixed positive integer and $\Omega \subset \R^n$ is an open set. In this section, we introduce fundamental properties of function spaces and operators which will be used in our study. 
	
	\subsection{Fractional Sobolev spaces and fractional Laplacian}
	\label{sec: Fractional Sobolev spaces}
	
	We denote by $\schwartz(\R^n)$ and $\tempered(\R^n)$ Schwartz functions and tempered distributions respectively. We define the Fourier transform by
	\begin{equation}
		\fourier u(\xi) \vcentcolon = \int_{\R^n} u(x)e^{-\mathrm{i}x \cdot \xi} \,dx,
	\end{equation}
	which is occasionally also denoted by $\widehat{u}$, where $\mathrm{i}=\sqrt{-1}$. By duality the Fourier transform can be extended to the space of tempered distributions where it will again be denoted by $\fourier u = \widehat{u}$, where $u \in \tempered(\R^n)$. We denote the inverse Fourier transform by $\ifourier$. 
	
	Given $s\in\R$, the $L^2$-based fractional Sobolev space $H^{s}(\R^{n})$ is the set of all tempered distributions $u\in\tempered(\R^n)$ such that
	\begin{equation}\notag
		\|u\|_{H^{s}(\mathbb{R}^{n})}\vcentcolon = \left\|\langle D\rangle^s u \right\|_{L^2(\R^n)}<\infty,
	\end{equation}
	where $\langle D\rangle^s$ is the Bessel potential operator of order $s$ with Fourier symbol $$\LC 1+|\xi|^2\RC^{s/2}.$$ 
	The fractional Laplacian of order $s\geq 0$ can be defined as the Fourier multiplier
	\begin{equation}\label{eq:fracLapFourDef}
		(-\Delta)^{s} u = \ifourier\LC \abs{\xi}^{2s}\widehat{u}(\xi)\RC,
	\end{equation}
	for $u \in \tempered(\R^n)$ whenever the right-hand side of the above identity is well-defined. 
	In addition, it is also known that for $s\geq 0$, an equivalent norm on $H^s(\R^n)$ is given by
	\begin{equation}
		\label{eq: equivalent norm on Hs}
		\|u\|_{H^s(\R^n)}^*= \|u\|_{L^{2}(\mathbb{R}^{n})}+\left\|(-\Delta)^{s/2}u  \right\|_{L^{2}(\mathbb{R}^{n})},
	\end{equation}
	and the fractional Laplacian $(-\Delta)^{s}\colon H^{t}(\R^n) \to H^{t-2s}(\R^n)$ is a bounded linear operator for all $s \geq0$ and $t \in \R$. The next results assert the unique continuation property (UCP) and a suitable Poincar\'e inequality for the fractional Laplacian on bounded domains $\Omega\subset\R^n$.
	
	\begin{proposition}[{UCP for fractional Laplacians}]\label{prop:UCP} 
		Let $s> 0$ be a non-integer and $t\in\R$. If $u\in H^t(\R^n)$ satisfies $u=(-\Delta)^s u=0$ in a nonempty open subset $V\subset \R^n$, then $u\equiv 0$ in $\R^n$.
	\end{proposition}
	
	The preceding proposition was first shown in \cite[Theorem 1.2]{GSU20} for the case $s\in (0,1)$, in which case the fractional Laplacian $(-\Delta)^s$ can be equivalently computed as the singular integral (up to a constant depending on $n\in \N$ and $s\in (0,1)$) 
	\[
	(-\Delta)^su(x)=\text{p.v.}\int_{\R^n}\frac{u(x)-u(y)}{|x-y|^{n+2s}}\,dy
	\]
	for sufficiently nice functions $u$, where p.v. stands for the Cauchy principal value. For the higher order case $s>1$, one can apply the standard Laplacian to the equation, then the classical UCP for the Laplacian yields iteratively the desired result. With UCP at hand, one may derive a remarkable Runge approximation for nonlocal equations, which was first observed by \cite[Theorem 1.3]{GSU20}. In addition, we will prove Runge approximation for a nonlocal wave equation in Proposition \ref{prop: runge}.

	\begin{proposition}[{Poincar\'e inequality (cf.~\cite[Lemma~5.4]{RZ-unbounded})}]\label{prop:Poincare ineq} Let $\Omega\subset\R^n$ be a bounded domain. For any $s\geq 0$, there exists $C>0$ such that
		\begin{equation}
			\label{eq: poincare ineq}
			\|u\|_{L^2(\Omega)}\leq C  \left\|(-\Delta)^{s/2}u  \right\|_{L^2(\R^n)}
		\end{equation}
		for all $u\in C_c^{\infty}(\Omega)$.
	\end{proposition}
	
	Next we introduce some local variants of the above fractional Sobolev spaces. If $\Omega\subset \R^n$ is an open set, $F\subset \R^n$ a closed set and $s\in\mathbb{R}$,
	then we set
	\begin{align*}
		H^{s}(\Omega) & \vcentcolon =\left\{ u|_{\Omega} : \, u\in H^{s}(\mathbb{R}^{n})\right\},\\
		\widetilde{H}^{s}(\Omega) & \vcentcolon =\text{closure of \ensuremath{C_{c}^{\infty}(\Omega)} in \ensuremath{H^{s}(\mathbb{R}^{n})}},\\
		H_F^s&\vcentcolon =\{u\in H^s(\R^n)\,;\,\supp(u)\subset F\}.
	\end{align*}
	Meanwhile, $H^{s}(\Omega)$ is a Banach space with respect to the quotient norm
	\[
	\|u\|_{H^{s}(\Omega)}\vcentcolon =\inf\left\{ \|U\|_{H^{s}(\mathbb{R}^{n})} : \,  U\in H^{s}(\mathbb{R}^{n})\mbox{ and }U|_{\Omega}=u\right\}.
	\]
	Hence, using the fact that \eqref{eq: equivalent norm on Hs} is an equivalent norm on $\widetilde{H}^s(\Omega)$, Propositions~\ref{prop:Poincare ineq} and the density of $C_c^{\infty}(\Omega)$ in $\widetilde{H}^s(\Omega)$, we have:
	\begin{lemma}
		\label{lemma: equivalent norm on tilde spaces}
		Let $\Omega\subset\R^n$ be a bounded domain and $s\geq 0$. Then an equivalent norm on $\widetilde{H}^s(\Omega)$ is given by
		\begin{equation}
			\label{eq: equivalent norm on tilde spaces}
			\|u\|_{\widetilde{H}^s(\Omega)}=\left\|(-\Delta)^{s/2}u \right\|_{L^2(\R^n)}.
		\end{equation}
	\end{lemma}
	The observation of Lemma~\ref{lemma: equivalent norm on tilde spaces} will be of constant use in the well-posedness theorem below.
	
	\subsection{Bochner spaces}
	\label{subsec: Bochner spaces}
	Next, we introduce some standard function spaces for time-dependent PDEs adapted to the nonlocal setting considered in this article. First if $X$ is a given Banach space and $(a,b)\subset\R$, then we let $C^k([a,b]\,;X)$, $L^p(a,b\,;X)$ ($k\in\N,1\leq p<\infty$) stand for the space of $k$-times continuously differentiable functions and the space of measurable functions $u\colon (a,b)\to X$ such that $t\mapsto \|u(t)\|_X\in L^p([a,b])$ (with the usual modification for $p=\infty$). These spaces carry the norms
	\begin{equation}
		\label{eq: Bochner spaces}
		\begin{split}
			\|u\|_{L^p(a,b\,;X)}&\vcentcolon = \left(\int_a^b\|u(t)\|_{X}^p\,dt\right)^{1/p}<\infty,\\
			\|u\|_{C^k([a,b];X)}&\vcentcolon = \sup_{0\leq \ell\leq k}\|\partial_t^{\ell}u\|_{L^{\infty}(a,b;X)}.
		\end{split}
	\end{equation} 
	
	Additionally, whenever $u\in L^1_{\loc}(a,b\,;X)$ with $X$ being a space of functions over a subset of some euclidean space, such as $L^2(\Omega)$ or $H^s(\R^n)$, then $u$ is identified with a function $u(x,t)$ and $u(t)$ denotes the function $ x\mapsto u(x,t)$ for almost all $t$. This is justified by the fact, that any $u\in L^q(a,b\,;L^p(\Omega))$ with $1\leq q,p<\infty$ can be seen as a measurable function $u\colon \Omega\times (a,b)\to \R$ such that the norm $\|u\|_{L^q(a,b\,;L^p(\Omega))}$, as defined in \eqref{eq: Bochner spaces}, is finite. In particular, one has $L^p(0,T;L^p(\Omega))=L^p(\Omega_T)$ for $1\leq p<\infty$. Recall that we denote $A_T=A\times (0,T)$ for any set $A\subset\R^n$ and $T>0$. Clearly, a similar statement holds for the spaces $L^q(a,b\,;H^{s}(\R^n))$ and their local versions. Furthermore, the distributional derivative $\frac{du}{dt}\in \distr((a,b)\,;X)$ is identified with the derivative $\partial_tu\in \distr(\Omega\times (a,b))$ as long as it is well-defined. Here $ \distr((a,b)\,;X)$ stands for all continuous linear operators from $C_c^{\infty}((a,b))$ to $X$.
	
	\section{The forward problem of nonlinear nonlocal wave equations}
	\label{sec: well-posedness}
	The main purpose of this section is to establish the well-posedness of the semilinear nonlocal wave equation for certain nonlinearities $f(u)$. As a preliminary step we first show the well-posedness of the linear nonlocal wave equation with a potential $q$ only belonging to some $L^p$ space.
	
	\subsection{Well-posedness}
	
	Let us start by stating the well-posedness theorem for the linear nonlocal wave equation.
	
	\begin{theorem}[Well-posedness of linear nonlocal wave equation]
		\label{thm: well-posedness linear}
		Let $\Omega\subset\R^n$ be a bounded Lipschitz domain, $T>0$ and $s>0$ a non-integer. Suppose that $q\in L^p(\Omega)$, where $p$ satisfies
		\[
		\begin{cases}
			n/s\leq p\leq \infty, &\, \text{if }\, 2s< n,\\
			2<p\leq \infty,  &\, \text{if }\, 2s= n,\\
			2\leq p\leq \infty, &\, \text{if }\, 2s\geq n.
		\end{cases}
		\]
		Then for any pair $(u_0, u_1)\in \widetilde{H}^s(\Omega)\times L^2(\Omega)$ and $F\in L^2(\Omega_T)$ there exists a unique solution $u\in C([0,T];\widetilde{H}^s(\Omega))$ with $\partial_t u\in C([0,T];L^2(\Omega))$ of
		\begin{equation}
			\label{eq: well-posedness linear case}
			\begin{cases}
				\LC \partial_t^2 +(-\Delta)^s+q\RC u=F &\text{ in }\Omega_T,\\
				u=0     &\text{ in }(\Omega_e)_T,\\
				u(0)=u_0,\quad \partial_t u(0)=u_1  &\text{ in }\Omega.
			\end{cases}
		\end{equation}
		More concretely, this means that there holds
		\begin{equation}
			\label{eq: weak solution}
			\frac{d}{dt}\left\langle\partial_t u,v \right\rangle_{L^2(\Omega)}+ \left\langle (-\Delta)^{s/2}u,(-\Delta)^{s/2}v \right\rangle_{L^2(\R^n)}+\left\langle qu,v \right\rangle_{L^2(\Omega)}=\langle F,v\rangle_{L^2(\Omega)}
		\end{equation}
		for all $v\in \widetilde{H}^s(\Omega)$ in the sense of $\distr((0,T))$. Furthermore, $u$ satisfies the following energy identity
		\begin{equation}
			\label{eq: energy identity linear case}
			\begin{split}
				& \quad \,  \left\|\partial_t u(t) \right\|_{L^2(\Omega)}^2+ \left\|(-\Delta)^{s/2}u(t)\right\|_{L^2(\R^n)}^2+2 \left\langle q u,\partial_t u \right\rangle_{L^2(\Omega_t)}\\
				&= \left\|u_1 \right\|_{L^2(\Omega)}^2+\left\|(-\Delta)^{s/2}u_0 \right\|_{L^2(\R^n)}^2+2\left\langle F,\partial_t u \right\rangle_{L^2(\Omega_t)}
			\end{split}
		\end{equation}
		for all $t\in [0,T]$. Moreover, if $(u_{0,j},u_{1,j})\in \widetilde{H}^s(\Omega)\times L^2(\Omega)$, $F_j\in L^2(\Omega_T)$ and $u_j$ denotes the related unique solution to \eqref{eq: weak solution} for $j=1,2$, then the following continuity estimate holds
		\begin{equation}
			\label{eq: continuity estimate}
			\begin{split}
				&\quad \, \left\|u_1-u_2 \right\|_{L^{\infty}(0,T;\widetilde{H}^s(\Omega))}+ \left\|\partial_t u_1-\partial_t u_2 \right\|_{L^{\infty}(0,T;L^2(\Omega))}\\
				& \leq C\LC \left\|u_{0,1}-u_{0,2} \right\|_{\widetilde{H}^s(\Omega)}+\left\|u_{1,1}-u_{1,2} \right\|_{L^2(\Omega)}+ \left\|F_1-F_2 \right\|_{L^2(\Omega_T)}\RC, 
			\end{split}
		\end{equation}
		for some $C>0$ depending on $T>0$.
	\end{theorem}
	\begin{remark}
		\label{remark: nonhom linear problem}
		Notice that the previous well-posedness result can easily be used to construct a unique solution $u\in C([0,T];H^s(\R^n))$ with $\partial_tu\in C([0,T];L^2(\R^n))$ to the problem
		\begin{equation}
			\label{eq: well-posedness linear case non-hom}
			\begin{cases}
				\LC \partial_t^2 +(-\Delta)^s+q\RC u=f &\text{ in }\Omega_T,\\
				u=\varphi  &\text{ in }(\Omega_e)_T,\\
				u(0)=u_0,\quad \partial_t u(0)=u_1  &\text{ in } \Omega,
			\end{cases}
		\end{equation}
		whenever the initial conditions, force term, potential are as above and the exterior Dirichlet data $\varphi$ is sufficiently regular, say $\varphi\in C^2([0,T];H^{2s}(\R^n))$. Furthermore, \eqref{eq: continuity estimate} and the mapping properties of the fractional Laplacian directly imply the continuity estimate
		\begin{equation}
			\label{eq: continuity estimate with non-zero ext cond}
			\begin{split}
				&\quad \, \|u_1-u_2\|_{L^{\infty}(0,T;H^s(\R^n))}+\left\|\partial_t u_1-\partial_t u_2 \right\|_{L^{\infty}(0,T;L^2(\R^n))}\\
				& \leq C\Big( \left\|u_{0,1}-u_{0,2}\right\|_{H^s(\R^n)}+\left\|u_{1,1}-u_{1,2}\right\|_{L^2(\R^n)}+\left\|F_1-F_2\right\|_{L^2(\Omega_T)}\\
				&\qquad \quad  + \left\|\varphi_1-\varphi_2 \right\|_{C^2([0,T];H^{2s}(\R^n))}\Big).
			\end{split}
		\end{equation}
		Here, $(u_{0,j},u_{1,j})\in H^s(\R^n)\times L^2(\R^n)$, $F_j\in L^2(\Omega)$, $\varphi_j\in C^2([0,T];\widetilde{H}^{2s}(\Omega_e))$ with $u_{0,j}-\varphi_j(0)\in\widetilde{H}^s(\Omega), u_{1,j}-\partial_t\varphi_j(0)\in L^2(\Omega)$ and $u_j$ denotes the related unique solution to
		\begin{equation}
			\label{eq: well-posedness linear case j}
			\begin{cases}
				\LC \partial_t^2 +(-\Delta)^s+q\RC u=F_j &\text{ in }\Omega_T,\\
				u=\varphi_j     &\text{ in }(\Omega_e)_T,\\
				u(0)=u_{0,j},\quad \partial_t u(0)=u_{1,j}  &\text{ in }\Omega.
			\end{cases}
		\end{equation}
	\end{remark}
	
	\begin{proof}
		Let us define
		\begin{equation}
			\label{eq: functional setup}
			V=\widetilde{H}^s(\Omega),\quad H=L^2(\Omega),\quad V'=(\widetilde{H}^s(\Omega))',
		\end{equation}
		where $V$ is endowed with the equivalent norm from Lemma~\ref{lemma: equivalent norm on tilde spaces},
		and
		\begin{equation}
			\label{eq: sesquilinear forms}
			a_0(u,v)=\left\langle (-\Delta)^{s/2}u,(-\Delta)^{s/2}v \right\rangle_{L^2(\R^n)},\quad a_1(u,v)= \langle qu,v\rangle_{L^2(\Omega)}.
		\end{equation}
		for $u,v\in V$. Observe that $a_0$ and $a_1$ are continuous sesquilinear forms on $\widetilde{H}^s(\Omega)$ and furthermore there holds
		\begin{equation}
			\label{eq: estimate a1}
			\left|\langle qu,v\rangle_{L^2(\Omega)} \right|\leq C\|q\|_{L^p(\Omega)}\|u\|_{\widetilde{H}^s(\Omega)}\|v\|_{L^2(\Omega)}
		\end{equation}
		for some $C>0$.  The case $p=\infty$ is clear.  In the case $\frac{n}{s}\leq p<\infty$ with $2s< n$ one can use H\"older's inequality with
		\[
		\frac{1}{2}=\frac{n-2s}{2n}+\frac{s}{n},
		\]
		$L^{r_2}(\Omega)\hookrightarrow L^{r_1}(\Omega)$ for $r_1\leq r_2$ as $\Omega\subset\R^n$ is bounded and Sobolev's inequality to obtain
		\begin{equation}
			\label{eq: computation for L2 estimate}
			\begin{split}
				\left|\langle qu,v\rangle_{L^2(\Omega)} \right|&\leq \|qu\|_{L^2(\Omega)}\|v\|_{L^2(\Omega)}\\
				&\leq \|q\|_{L^{n/s}(\Omega)}\|u\|_{L^{\frac{2n}{n-2s}}(\Omega)}\|v\|_{L^2(\Omega)}\\
				&\leq C\|q\|_{L^{n/s}(\Omega)}\|u\|_{L^{\frac{2n}{n-2s}}(\Omega)}\|v\|_{L^2(\Omega)}\\
				&\leq C\|q\|_{L^{p}(\Omega)}\left\|(-\Delta)^{s/2} u \right\|_{L^2(\R^n)}\|v\|_{L^2(\Omega)}
			\end{split}
		\end{equation}
		for all $u,v\in \widetilde{H}^s(\Omega)$.

		In the case $2s>n$ one can use the embedding $H^s(\R^n)\hookrightarrow L^{\infty}(\R^n)$ together with Lemma~\ref{lemma: equivalent norm on tilde spaces} and the boundedness of $\Omega$ to see that the estimate \eqref{eq: computation for L2 estimate} holds. In the case $n=2s$ one can use the boundedness of the embedding $\widetilde{H}^s(\Omega)\hookrightarrow L^{\bar{p}}(\Omega)$ for all $2\leq \bar{p}<\infty$, H\"older's inequality and the boundedness of $\Omega$ to get the estimate \eqref{eq: computation for L2 estimate}. In fact, the aforementioned embedding in the critical case follows by \cite{Ozawa} and the Poincar\'e inequality. Next observe that the spaces $V,H,V'$ and sesquilinear forms $a_0,a_1$ as given in \eqref{eq: functional setup},   \eqref{eq: sesquilinear forms}, respectively, satisfy the assumptions in \cite[Chapter~XVIII, \S 5, Sections 1.1,1.2,5.2]{DautrayLionsVol5}. Therefore, we can apply \cite[Chapter~XVIII, \S 5, Theorem 3 and 4]{DautrayLionsVol5} to deduce that there exists a unique solution $u$ of \eqref{eq: well-posedness linear case} in the sense that \eqref{eq: weak solution} holds. Note that a priori this solution can be complex, but since we always take real-valued data in our problems we know that $u$ is real-valued as well. Finally, the energy identity is nothing else than a rephrasing of \cite[Chapter~XVIII, \S 5, Lemma 7]{DautrayLionsVol5}. 
		
		The continuity estimate \eqref{eq: continuity estimate with non-zero ext cond} follows from the energy identity, since in this case we have
		\begin{equation}
			\begin{split}
				& \quad \, \left\|\partial_t u_1(t)-\p_t u_2(t) \right\|_{L^2(\Omega)}^2+ \left\|u_1(t)-u_2(t)\right\|_{\wt H^s(\Omega)}^2 \\
				&\quad \, +2 \left\langle q(u_1-u_2),\partial_t (u_1-u_2) \right\rangle_{L^2(\Omega_t)}\\
				&= \left\|u_{1,1}-u_{1,2} \right\|_{L^2(\Omega)}^2+\left\|u_{0,1}-u_{0,2} \right\|_{\wt H^s(\Omega)}^2\\
				&\quad \, +2\left\langle F_1-F_2,\partial_t (u_1-u_2)\right\rangle_{L^2(\Omega_t)}.
			\end{split}
		\end{equation}
		Using Cauchy-Schwarz, Young's inequality, and the estimate \eqref{eq: computation for L2 estimate} we obtain
		\begin{equation}
			\begin{split}
				& \quad \,  \left\|\p_t u_1(t)-\p_t u_2(t) \right\|_{L^2(\Omega)}^2+ \left\|u_1(t)-u_2(t)\right\|_{\wt H^s(\Omega)}^2 \\
				&\leq C\bigg(
				\left\|u_{1,1}-u_{1,2} \right\|_{L^2(\Omega)}^2+\left\|u_{0,1}-u_{0,2} \right\|_{\wt H^s(\Omega)}^2\\
				&\qquad \quad 
				+\Vert F_1-F_2\Vert_{L^2(\Omega_T)}^2
				+
				\int_0^t\Vert
				\p_t u_1(\rho)-\p_tu_2(\rho) \Vert_{L^2(\Omega)}^2\,d\rho \bigg).
			\end{split}
		\end{equation}
		Finally, Gr\"onwall's inequality and taking supermum over $[0,T]$ yield the claimed continuity property.
		Thus we can conclude the proof.
	\end{proof}
	
	Next, we move on to the nonlinear problem. We consider here slightly more general setting than in equation \eqref{eq: nonlinear wave equation}, containing also a nonlinear term $g(x,\p_t u)$. For this purpose let us specify more precisely the assumptions we make on the nonlinearities $f$ and $g$. We start by recalling the notion of a Carath\'eodory function.
	
	\begin{definition}
		Let $U\subset \R^n$ be an open set. We say that  $f\colon U\times\R\to\R$ is a \emph{Carath\'odory function}, if it has the following properties:
		\begin{enumerate}[(i)]
			\item $\tau\mapsto f(x,\tau)$ is continuous for a.e. $x\in U$,
			\item $x\mapsto f(x,\tau)$ is measurable for all $\tau \in\R$.
		\end{enumerate}
	\end{definition}
	
	\begin{assumption}\label{main assumptions on nonlinearities}
		Let $f\colon \Omega\times \R\to\R$ and $g\colon \Omega\times \R\to\R$ be two Carath\'eodory functions satisfying the following conditions:
		\begin{enumerate}[(i)]
			\item\label{prop f} $f$ has partial derivative $\partial_{\tau}f$, which is a Carath\'eodory function, and there exists $a\in L^p(\Omega)$ such that
			\begin{equation}
				\label{eq: bound on derivative}
				\left|\partial_\tau f(x,\tau)\right|\lesssim a(x)+|\tau|^r
			\end{equation}
			for all $\tau\in\R$ and a.e. $x\in\Omega$\footnote{The symbols $\lesssim$ denotes the inequality holds up to a positive constant whose value is irrelevant for our arguments.}. Here the exponents $p$ and $r$ satisfy the restrictions
			\begin{equation}\label{eq: cond on p}
				\begin{cases}
					n/s\leq p\leq \infty, &\, \text{if }\, 2s< n,\\
					2<p\leq \infty,  &\, \text{if }\, 2s= n,\\
					2\leq p\leq \infty, &\, \text{if }\, 2s\geq n,
				\end{cases}
			\end{equation} 
			and
			\begin{equation}
				\label{eq: cond on r}
				\begin{cases}
					0\leq r<\infty, &\, \text{if }\, 2s\geq n,\\
					0\leq r\leq \frac{2s}{n-2s}, &\, \text{if }\, 2s< n,
				\end{cases}
			\end{equation}
			respectively. Moreover, $f$ fulfills the integrability condition $f(\cdot,0)\in L^2(\Omega)$.
			\item\label{prop F} There is a constant $C_1>0$ such that the function $F\colon \Omega\times\R\to\R$ defined via
			\[
			F(x,\tau)=\int_0^\tau f(x,\rho)\,d\rho
			\]
			satisfies $F(x,\tau)\geq -C_1$ for all $\tau\in\R$ and $x\in\Omega$.
			\item\label{prop g} $g$ is uniformly Lipschitz continuous in $x\in\Omega$ with $g(\cdot,0)\in L^2(\Omega)$.

		\end{enumerate}
	\end{assumption}
	
	\begin{remark}
		Let $q\in L^\infty(\Omega)$ be a non-negative function. An example of a nonlinearity $f$, which satisfies the conditions in Assumption~\ref{main assumptions on nonlinearities} is given by a fractional power type nonlinearity $f(x,\tau)=q(x)|\tau|^r\tau$ for $r\geq 0$, which satisfies \eqref{eq: cond on r}. The regularity conditions are clearly fulfilled. Moreover, one easily checks that there holds
		\begin{equation}\label{eq: f-power 1}
			\p_\tau f(x,\tau)=q(x)(r+1)|\tau|^r
		\end{equation}
		and 
		\begin{equation}\label{eq: f-power 2}
			\frac{d}{d\tau}|\tau|^{r+2}=(r+2)|\tau|^r \tau
		\end{equation}
		for $\tau\neq 0$. It is not hard to see that the first identity \eqref{eq: f-power 1} implies \eqref{eq: bound on derivative}. On the other hand, from the second identity \eqref{eq: f-power 2}, we can get
		\[
		\begin{split}
			F(x,\tau)&=q(x)\int_0^{\tau}|\rho|^{r}\rho\,d\rho=\frac{q(x)}{r+2}\int_0^{\tau}\frac{d}{d\rho}|\rho|^{r+2}\,d\rho=q(x)\frac{|\tau|^{r+2}}{r+2}\geq 0
		\end{split}
		\]
		and hence the condition \ref{prop F} is satisfied as well. We also point out the condition \eqref{eq: cond on p} of the exponent $p$ satisfies the assumption in Theorem \ref{thm: well-posedness linear}.
	\end{remark}
	
	\begin{theorem}[Well-posedness of nonlinear nonlocal wave equation]
		\label{thm: well-posedness nonlinear}
		Let $\Omega\subset\R^n$ be a bounded Lipschitz domain, $T>0$ and $s>0$ a non-integer. Suppose that $q,f,F$ and $g$ satisfy Assumption~\ref{main assumptions on nonlinearities}. Then for any pair $(u_0,u_1)\in \widetilde{H}^s(\Omega)\times L^2(\Omega)$ and $h\in L^2(\Omega_T)$ there exists a unique solution 
		\begin{equation}
			u\in C([0,T];\widetilde{H}^s(\Omega)) \quad \text{and}\quad \partial_t u\in C([0,T];L^2(\Omega))
		\end{equation} 
		of
		\begin{equation}
			\label{eq: well-posedness nonlinear case}
			\begin{cases}
				\partial_t^2u +(-\Delta)^su+f(x,u)+g(x,\partial_t u)=h &\text{ in }\Omega_T,\\
				u=0&\text{ in }(\Omega_e)_T,\\
				u(0)=u_0,\quad \partial_t u(0)=u_1  &\text{ in }\Omega.
			\end{cases}
		\end{equation}
		More concretely, this means that there holds
		\begin{equation}
			\label{eq: weak solution nonlin}
			\begin{split}
				&\frac{d}{dt}\left\langle\partial_t u,v\right\rangle_{L^2(\Omega)}+\left\langle (-\Delta)^{s/2}u,(-\Delta)^{s/2}v\right\rangle_{L^2(\R^n)}\\
				&\quad+\langle f(\cdot,u),v\rangle_{L^2(\Omega)}+\left\langle g(\cdot,\partial_t u),v \right\rangle_{L^2(\Omega)}=\langle h,v\rangle_{L^2(\Omega)}
			\end{split}
		\end{equation}
		for all $v\in \widetilde{H}^s(\Omega)$ in the sense of $\distr((0,T))$. Furthermore, $u$ satisfies the following energy identity
		\begin{equation}
			\label{eq: energy identity nonlinear case}
			\begin{split}
				&\quad\,  \left\|\partial_t u(t)\right\|_{L^2(\Omega)}^2+\left\|(-\Delta)^{s/2}u(t)\right\|_{L^2(\R^n)}^2 \\
				&\quad \, +2\int_{\Omega}F(x,u(t))dx+2\left\langle g(\cdot,\partial_tu),\partial_t u\right\rangle_{L^2(\Omega_t)}\\
				&= \|u_1\|_{L^2(\Omega)}^2+\left\|(-\Delta)^{s/2}u_0\right\|_{L^2(\R^n)}^2+2\left\langle h,\partial_t u\right\rangle_{L^2(\Omega_t)}+2\int_{\Omega}F(x,u_0)\,dx
			\end{split}
		\end{equation}
		for all $t\in [0,T]$.
	\end{theorem}
	
	\begin{proof}
		Let $\psi\in C([0,T];L^2(\Omega))$. First we show that $g(x,\psi)$, $f(\cdot,\psi)$ are in $L^2(\Omega_T)$. As $g$ is a Carath\'eodory function the composition $g(x,\psi(x,t))$ is measurable for every $t\in [0,T]$. Additionally, the Lipschitz continuity of $g(x,\cdot)$ directly implies that 
		\begin{equation}
			\label{eq: comp with g}
			|g(x,\psi(x,t))|\leq C(|g(x,0)|+|\psi(x,t)|)
		\end{equation}
		and we have
		\begin{equation}
			\label{eq: L2 estimate comp with g}
			\|g(\cdot,\psi)\|_{L^2(\Omega_T)}\leq CT^{1/2}\LC \|g(\cdot,0)\|_{L^2(\Omega)}+\|\psi\|_{L^{\infty}(0,T;L^2(\Omega))}\RC,
		\end{equation}
		for some constant $C>0$. By the same argument as for $g$ the composition $f(x,\psi(x,t))$ is measurable for every $t\in [0,T]$. Moreover, by the fundamental theorem of calculus and Assumption~\ref{main assumptions on nonlinearities} we have
		\begin{equation}
			\label{eq:estimate on f}
			|f(x,s)-f(x,0)|\leq \left|\int_0^s \partial_{\tau}f(x,\tau)\,d\tau\right|\leq C\LC |a(x)||s|+|s|^{r+1}\RC 
		\end{equation}
		for a.e. $x\in\Omega$, $s\in\R$. This implies
		\[
		|f(x,\psi(x,t))|\leq C\LC |f(x,0)|+|a(x)||\psi(x,t)|+|\psi(x,t)|^{r+1}\RC.
		\]
		This guarantees that
		\begin{equation}
			\label{eq: L2 estimate f}
			\|f(\cdot, \psi(t))\|_{L^2(\Omega)}\leq C\LC \|f(\cdot,0)\|_{L^2(\Omega)}+\|a\psi(t)\|_{L^2(\Omega)}+\|\psi(t)\|_{L^{2(r+1)}(\Omega)}^{r+1}\RC,
		\end{equation}
		for $0\leq t\leq T$. From now on let us additionally assume that $\psi\in C([0,T];H^s(\R^n))$ and without loss of generality we can assume $r>0$. Then the computation in \eqref{eq: computation for L2 estimate} allows to estimate 
		\begin{equation}
			\label{eq: linear growth term}
			\|a\psi(t)\|_{L^2(\Omega)}\leq C\|a\|_{L^p(\Omega)}\left\|(-\Delta)^{s/2}\psi(t)\right\|_{L^2(\R^n)}.
		\end{equation}
		Next note that 
		\[
		\begin{cases}
			1\leq 1+r<\infty, &\, \text{if }\, 2s\geq n,\\
			1\leq 1+r\leq \frac{n}{n-2s} &\, \text{if }\, 2s< n.
		\end{cases}
		\]
		If $2s> n$, then the Sobolev embedding $H^s(\R^n)\hookrightarrow L^{\infty}(\R^n)$ implies
		\begin{equation}
			\label{eq: supercritical case}
			\|\psi(t)\|_{L^{2(r+1)}(\Omega)}^{r+1}\leq C\|\psi(t)\|_{H^s(\R^n)}^{r+1}.
		\end{equation}
		In the critical case $2s=n$, we can apply \cite{Ozawa} to obtain
		\begin{equation}
			\label{eq: critical case}
			\|\psi(t)\|_{L^{2(r+1)}(\Omega)}^{r+1}\leq C\|(-\Delta)^{s/2}\psi(t)\|_{L^2(\R^n)}^r\|\psi(t)\|_{L^2(\R^n)}.
		\end{equation}
		In the subcritical case $2s<n$, we apply the Hardy--Littlewood-Sobolev lemma to deduce
		\begin{equation}
			\label{eq: subcritical case}
			\|\psi(t)\|_{L^{2(r+1)}(\Omega)}^{r+1}\leq C\|\psi(t)\|_{L^{\frac{2n}{n-2s}}(\Omega)}^{r+1}\leq C \left\|(-\Delta)^{s/2}\psi(t) \right\|_{L^2(\R^n)}^{r+1}.
		\end{equation}
		As $\psi\in C([0,T];H^s(\R^n))$, we get by the continuity of the fractional Laplacian the estimate
		\begin{equation}
			\label{eq: estimate Hs}
			\begin{split}
				\|f(\cdot,\psi)\|_{L^2(\Omega_T)}
				\lesssim\, &T^{1/2}\bigg(\|f(\cdot,0)\|_{L^2(\Omega)}+\|a\|_{L^p(\Omega)}\|\psi\|_{L^{\infty}(0,T;H^s(\R^n))}\\
				& \qquad \quad +\|\psi\|^{r+1}_{L^{\infty}(0,T;H^s(\R^n))}\bigg).
			\end{split}
		\end{equation}
		Next, let us set 
		\begin{equation}
			\label{eq: size initial cond}
			A\vcentcolon = \max \LC \|u_0\|_{\widetilde{H}^s(\Omega)},\|u_1\|_{L^2(\Omega)}\RC, 
		\end{equation}
		and for constants $0<T_0\leq T$ and $R\geq A$, which will be fixed later. Consider the space 
		\begin{equation}
			\label{eq: solution space}
			X_{T_0,R}:=\left\{u\in C([0,T_0];\widetilde{H}^s(\Omega))\cap C^1([0,T_0];L^2(\Omega)) :\,  \|u\|_{T_0}\leq R \right\},
		\end{equation}
		where $\|\cdot\|_{T_0}$ is given by
		\[
		\|u\|_{T_0}:=\max\LC \|u\|_{L^{\infty}(0,T_0;\widetilde{H}^s(\Omega))}, \,  \left\|\partial_t u\right\|_{L^{\infty}(0,T_0;L^2(\Omega))}\RC.
		\]
		It is a known fact that under the given norm $C([0,T_0];\widetilde{H}^s(\Omega))\cap C^1([0,T_0];L^2(\Omega))$ becomes a Banach space and thus the same holds for $X_{T_0,R}$. Now, for given $v\in X_{T_0,R}$, let us consider the linear problem  
		\begin{equation}
			\label{eq: linear part of nonlinear problem}
			\begin{cases}
				\LC \partial_t^2 +(-\Delta)^s \RC u=h-f(x,v)-g(x,\partial_tv) &\text{ in }\Omega_{T_0},\\
				u=0 &\text{ in }(\Omega_e)_{T_0},\\
				u(0)=u_0,\quad \partial_t u(0)=u_1 &\text{ in } \Omega.
			\end{cases}
		\end{equation}
		Theorem~\ref{thm: well-posedness linear} yields the well-posedness of \eqref{eq: linear part of nonlinear problem}, that is, there exists a unique solution $u\in C([0,T_0];\widetilde{H}^s(\Omega))\cap C^1([0,T_0];L^2(\Omega))$ to \eqref{eq: linear part of nonlinear problem}, and we can define a map 
		\begin{equation}
			S\colon X_{T_0,R}\to C([0,T_0];\widetilde{H}^s(\Omega))\cap C^1([0,T_0];L^2(\Omega)),
		\end{equation} 
		which maps any $v\in X_{T_0,R}$ to its unique solution $S(v)$ of \eqref{eq: linear part of nonlinear problem}. The rest of the proof is divided into four steps.\\
		
		\noindent{\textit{Step 1.}} First we show that there exists $R_0\geq A$ and for all $R\geq R_0$ there exists $T_0=T_0(R_0)>0$ such that $S(X_{T_0,R})\subset X_{T_0,R}$. 
		
		\medskip
		
		By the energy identity of Theorem~\ref{thm: well-posedness linear}, H\"older's inequality and Young's inequality, there holds
		\[
		\begin{split}
			&\quad \, \left\|\partial_t u(t)\right\|_{L^2(\Omega)}^2+\left\|(-\Delta)^{s/2}u(t)\right\|_{L^2(\R^n)}^2\\
			&= \left\|u_1\right\|_{L^2(\Omega)}^2+\left\|(-\Delta)^{s/2}u_0 \right\|_{L^2(\R^n)}^2+2 \left\langle h-f(\cdot,v)-g(\cdot,\partial_tv),\partial_t u \right\rangle_{L^2(\Omega_t)}\\
			&\leq \left\|u_1\right\|_{L^2(\Omega)}^2+ \left\|(-\Delta)^{s/2}u_0\right\|_{L^2(\R^n)}^2 +\frac{1}{2} \left\|\partial_t u \right\|_{L^{\infty}(0,T_0;L^2(\Omega))}^2\\
			&\quad \, +2T_0 \left\|h-f(\cdot,v)-g(\cdot,\partial_tv)\right\|_{L^2(\Omega_{T_0})}^2\\
			&\leq \left\|u_1\right\|_{L^2(\Omega)}^2+\left\|(-\Delta)^{s/2}u_0 \right\|_{L^2(\R^n)}^2 +\frac{1}{2} \left\|\partial_t u\right\|_{L^{\infty}(0,T_0;L^2(\Omega))}^2\\
			&\quad  \, +C T_0\LC \|h\|_{L^2(\Omega_{T_0})}^2+\|f(\cdot,v)\|_{L^2(\Omega_{T_0})}^2+\left\|g(\cdot,\partial_tv)\right\|_{L^2(\Omega_{T_0})}^2\RC\\
		\end{split}
		\]
		for all $0\leq t\leq T_0$ and $\eps>0$. Taking the supremum over $[0,T_0]$, absorbing the third term on the left-hand side and using \eqref{eq: size initial cond} we get
		\[
		\begin{split}
			&\quad \, \left\|\partial_t u \right\|_{L^{\infty}(0,T_0;L^2(\Omega))}^2+ \left\|(-\Delta)^{s/2}u \right\|_{L^{\infty}(0,T_0;L^2(\R^n))}^2\\
			&\leq 4A^2 +C T_0\LC \|h\|_{L^2(\Omega_{T_0})}^2+\|f(\cdot,v)\|_{L^2(\Omega_{T_0})}^2+\left\|g(\cdot,\partial_tv)\right\|_{L^2(\Omega_{T_0})}^2\RC .
		\end{split}
		\]
		By \eqref{eq: L2 estimate comp with g}, \eqref{eq: estimate Hs}, Lemma~\ref{lemma: equivalent norm on tilde spaces} and the definition of the space $X_{T_0,R}$, we get
		\[
		\begin{split}
			&\quad \, \left\|\partial_t u\right\|_{L^{\infty}(0,T_0;L^2(\Omega))}^2+ \left\|(-\Delta)^{s/2}u \right\|_{L^{\infty}(0,T_0;L^2(\R^n))}^2\\
			&\leq 4A^2 +CT_0\|h\|_{L^2(\Omega_{T})}^2\\
			&\quad+C T_0^2\LC \|f(\cdot,0)\|^2_{L^2(\Omega)}+\|a\|^2_{L^p(\Omega)}R^2+R^{2(r+1)}+\|g(\cdot,0)\|_{L^2(\Omega)}^2+R^2\RC.
		\end{split}
		\]
		This implies
		\[
		\begin{split}
			\|u\|_{T_0}&\leq 2A +C_0T_0^{1/2}+ C_1T_0(1+R+R^{r+1}),
		\end{split}
		\]
		where $C_0,C_1>0$ are constants only depending on $n,s,\Omega$ and the norms $\|f(\cdot,0)\|_{L^2(\Omega)}$, $\|g(\cdot,0)\|_{L^2(\Omega)}$, $\|h\|_{L^2(\Omega_T)}$ as well as the constants appearing in Assumption~\ref{main assumptions on nonlinearities}. By choosing 
		\begin{equation}
			\label{eq: choosing parameters}
			R_0\geq 4A\quad \text{and}\quad T_0\leq \min\left(\frac{R^2}{4C_0},\frac{R}{4C_1(1+R+R^{r+1})}\right),
		\end{equation}
		we get $\|u\|\leq R$, whenever $R\geq R_0$. Note that $T_0$ only depends on $n,s,\Omega$, the functions $f,g,h$ and $A$ measuring the size of the initial conditions. 
		\\
		
		\noindent{\textit{Step 2.}} Next we prove that by making $T_0(R)>0$ possibly smaller, the map $S\colon X_{T_0,R}\to X_{T_0,R}$ is a (strict) contraction. 
		
		\medskip
		
		Assume that $R_0$ and $T_0$ are chosen as in \eqref{eq: choosing parameters}. Let us start by observing the estimate
		\begin{equation}
			\label{eq: estimate difference of f}
			|f(x,t)-f(x,s)|\leq C\LC |a(x)|+|t|^r+|s|^r\RC |t-s|
		\end{equation}
		for all $s,t\in\R$ and a.e. $x\in\Omega$. Let $u^j=S(v^j)\in C([0,T_0];\widetilde{H}^s(\Omega))\cap C^1([0,T_0];L^2(\Omega))$ for $j=1,2$. Then we may estimate
		\begin{equation}
			\label{eq: estimate difference in f}
			\begin{split}
				&\quad \, \left\|f(\cdot,v^1(t))-f(\cdot,v^2(t))\right\|_{L^2(\Omega)}\\
				&\leq C\LC \left\|a(v^1(t)-v^2(t))\right\|_{L^2(\Omega)}+ \left\|(|v^1(t)|^r+|v^2(t)|^r)|v^1(t)-v^2(t)|\right\|_{L^2(\Omega)}\RC\\
				&\leq C\LC \|a\|_{L^p(\Omega)}\left\|v^1(t)-v^2(t)\right\|_{\widetilde{H}^s(\Omega)}+ \left\|(|v^1(t)|^r+|v^2(t)|^r)|v^1(t)-v^2(t)|\right\|_{L^2(\Omega)}\RC,
			\end{split}
		\end{equation}
		for all $0\leq t\leq T_0$, where we used in the second inequality \eqref{eq: computation for L2 estimate} and \eqref{eq: estimate difference of f}. Next suppose that $2s<n$, then H\"older's inequality with
		\[
		\frac{1}{2}=\frac{n-2s}{2n}+\frac{s}{n}
		\]
		and the Sobolev embedding imply
		\begin{equation}
			\label{eq: prlim estimate for contraction}
			\begin{split}
				&\quad \, \left\|(|v^1(t)|^r+|v^2(t)|^r)|v^1(t)-v^2(t)|\right\|_{L^2(\Omega)})\\
				&\leq \left\||v^1(t)|^r+|v^2(t)|^r \right\|_{L^{n/s}(\Omega)}\|v^1(t)-v^2(t)\|_{L^{\frac{2n}{n-2s}}(\Omega)}\\
				&\leq C\LC \left\||v^1(t)|^r \right\|_{L^{n/s}(\Omega)}+\left\||v^2(t)|^r\right\|_{L^{n/s}(\Omega)}\RC \left\|v^1(t)-v^2(t)\right\|_{\widetilde{H}^s(\Omega)}.
			\end{split}
		\end{equation}
		If $r\geq s/n$, then $1\leq r\leq \frac{2n}{n-2s}$ and thus the boundedness of $\Omega$ as well as the Sobolev embedding ensure that
		\begin{equation}
			\label{eq: subcritical case with r > sn}
			\begin{split}
				& \quad \, \left\|(|v^1(t)|^r+|v^2(t)|^r)|v^1(t)-v^2(t)|\right\|_{L^2(\Omega)})\\
				&\leq C\LC \left\|v^1(t)\right\|_{L^{r n/s}(\Omega)}^r+ \left\|v^2(t) \right\|^r_{L^{rn/s}(\Omega)}\RC \left\|v^1(t)-v^2(t) \right\|_{\widetilde{H}^s(\Omega)}\\
				&\leq C\LC \left\|v^1(t)\right\|_{L^{\frac{2n}{n-2s}}(\Omega)}^r+\left\|v^2(t)\right\|^r_{L^{\frac{2n}{n-2s}}(\Omega)}\RC \left\|v^1(t)-v^2(t)\right\|_{\widetilde{H}^s(\Omega)}\\
				&\leq C\LC \left\|v^1(t) \right\|_{\widetilde{H}^s(\Omega)}^r+ \left\|v^2(t)\right\|^r_{\widetilde{H}^s(\Omega)}\RC \left\|v^1(t)-v^2(t) \right\|_{\widetilde{H}^s(\Omega)},
			\end{split}
		\end{equation}
		for some constant $C>0$ independent of $v^1$ and $v^2$.

		Next assume that $0\leq r<s/n$. For $r=0$ the previous estimate still holds and so we can assume $0<r<s/n$. In this situation, we choose $z\geq 1$ such that $1\leq rz\leq 2$. This implies $n/s< 1/r\leq z$ and thus \eqref{eq: prlim estimate for contraction}, H\"older's inequality and Poincar\'e's inequality ensure
		\[
		\begin{split}
			&\quad \, \left\|(|v^1(t)|^r+|v^2(t)|^r)|v^1(t)-v^2(t)|\right\|_{L^2(\Omega)})\\
			&\leq C\LC \left\||v^1(t)|^r \right\|_{L^{n/s}(\Omega)}+\left\||v^2(t)|^r\right\|_{L^{n/s}(\Omega)}\RC \left\|v^1(t)-v^2(t)\right\|_{\widetilde{H}^s(\Omega)}\\
			&\leq C\LC \left\||v^1(t)|^r\right\|_{L^{z}(\Omega)}+\left\||v^2(t)|^r\right\|_{L^{z}(\Omega)}\RC \left\|v^1(t)-v^2(t)\right\|_{\widetilde{H}^s(\Omega)}\\
			&\leq C\LC \left\|v^1(t)\right\|^r_{L^{rz}(\Omega)}+\left\|v^2(t) \right\|^r_{L^{rz}(\Omega)}\RC\left\|v^1(t)-v^2(t)\right\|_{\widetilde{H}^s(\Omega)}\\
			&\leq C\LC \left\|v^1(t)\right\|^r_{L^{2}(\Omega)}+\left\|v^2(t)\right\|^r_{L^{2}(\Omega)}\RC \left\|v^1(t)-v^2(t)\right\|_{\widetilde{H}^s(\Omega)}\\
			&\leq C\LC \left\|v^1(t)\right\|^r_{\widetilde{H}^s(\Omega)}+\left\|v^2(t)\right\|^r_{\widetilde{H}^s(\Omega)}\RC \left\|v^1(t)-v^2(t)\right\|_{\widetilde{H}^s(\Omega)}.
		\end{split}
		\]
		Hence, for $2s<n$ we have for all $0\leq r\leq\frac{2}{n-2s}$ the estimate \eqref{eq: subcritical case with r > sn}. Next, suppose $2s>n$, then the Sobolev embedding $H^s(\R^n)\hookrightarrow L^{\infty}(\R^n)$ and the Poincar\'e inequality directly give
		\[
		\begin{split}
			&\quad \, \left\|(|v^1(t)|^r+|v^2(t)|^r)|v^1(t)-v^2(t)|\right\|_{L^2(\Omega)})\\
			&\leq C\LC \left\|v^1(t)\right\|_{L^{\infty}(\Omega)}^r+\left\|v^2(t)\right\|_{L^{\infty}(\Omega)}^r\RC \left\|v^1(t)-v^2(t)\right\|_{L^2(\Omega)}\\
			&\leq  C\LC \left\|v^1(t) \right\|_{\widetilde{H}^s(\Omega)}^r+\left\|v^2(t)\right\|_{\widetilde{H}^s(\Omega)}^r\RC \left\|v^1(t)-v^2(t)\right\|_{\widetilde{H}^s(\Omega)}
		\end{split}
		\]
		and so again obtain \eqref{eq: subcritical case with r > sn}. Finally, assume that the critical case $2s=n$ holds. If $r>0$, then we choose $z>2$ such that $rz> 2$, use the H\"older inequality with $2<p_0<\infty$ satisfying
		\[
		\frac{1}{2}=\frac{1}{z}+\frac{1}{p_0}
		\]
		and \cite{Ozawa} together with Poincar\'e's inequality to get
		\begin{equation}
			\label{eq: critical case contraction}
			\begin{split}
				& \quad \, \left\|\LC |v^1(t)|^r+|v^2(t)|^r\RC \left|v^1(t)-v^2(t)\right|\right\|_{L^2(\Omega)})\\
				&\leq C\LC \left\|v^1(t)\right\|^r_{L^{rz}(\Omega)}+\left\|v^2(t)\right\|^r_{L^{rz}(\Omega)}\RC \left\|v^1(t)-v^2(t)\right\|_{L^{p_0}(\Omega)}\\
				&\leq C\LC \left\|v^1(t)\right\|^r_{\widetilde{H}^s(\Omega)}+\left\|v^2(t)\right\|^r_{\widetilde{H}^s(\Omega)}\RC \left\|v^1(t)-v^2(t)\right\|_{\widetilde{H}^s(\Omega)}.
			\end{split}
		\end{equation}
		Therefore, also in this case we have the final estimate in \eqref{eq: subcritical case with r > sn}. The remaining case $r=0$ is immediate. Hence, we have in all cases \eqref{eq: subcritical case with r > sn} and therefore inserting this into \eqref{eq: estimate difference in f}, we deduce the bound 
		\begin{equation}
			\label{eq: l2 estimate of f for contraction}
			\begin{split}
				&\quad \, \left\|f(\cdot,v^1(t))-f(\cdot,v^2(t))\right\|_{L^2(\Omega)}\\
				& \leq C\LC \|a\|_{L^p(\Omega)}+\left\|v^1(t)\right\|^r_{\widetilde{H}^s(\Omega)}+\left\|v^2(t)\right\|^r_{\widetilde{H}^s(\Omega)}\RC \left\|v^1(t)-v^2(t) \right\|_{\widetilde{H}^s(\Omega)}.
			\end{split}
		\end{equation}
		Next, note that by Theorem~\ref{thm: well-posedness linear} the function $u=u^1-u^2$ is the unique solution to 
		\[
		\begin{cases}
			\LC \partial_t^2 +(-\Delta)^s\RC  u=-(f(x,v^1)-f(x,v^2))-(g(x,\partial_t v^1)-g(x,\partial_t v^2))&\text{ in }\Omega_T,\\
			u=0 &\text{ in }(\Omega_e)_T,\\
			u(0)=0,\quad \partial_t u(0)=0 &\text{ in }\Omega
		\end{cases}
		\]
		and satisfies the energy estimate 
		\[
		\begin{split}
			&\quad \, \left\|\partial_t u(t)\right\|_{L^2(\Omega)}^2+\left\|(-\Delta)^{s/2}u(t)\right\|_{L^2(\R^n)}^2\\
			&\leq C\LC \left\|f(\cdot,v^1)-f(\cdot,v^2) \right\|_{L^1(0,T_0;L^2(\Omega))}^2+ \left\|g(\cdot,\partial_t v^1)-g(\cdot,\partial_t v^2) \right\|^2_{L^1(0,T_0;L^2(\Omega))}\RC \\
			&\quad +\frac{1}{2}\left\|\partial_t u \right\|_{L^{\infty}(0,T_0;L^2(\Omega))}^2\\
			&\leq CT_0^2\LC \left\|f(\cdot,v^1)-f(\cdot,v^2)\right\|_{L^{\infty}(0,T_0;L^2(\Omega))}^2+\left\|g(\cdot,\partial_t v^1)-g(\cdot,\partial_t v^2)\right\|^2_{L^{\infty}(0,T_0;L^2(\Omega))}\RC \\
			&\quad +\frac{1}{2}\left\|\partial_t u \right\|_{L^{\infty}(0,T_0;L^2(\Omega))}^2
		\end{split}
		\]
		for all $0\leq t\leq T_0$.
		For the energy estimate we used additionally the H\"older and Young inequalities. Taking the supremum over $0\leq t\leq T_0$, absorbing the third term on the left-hand side, using the estimate \eqref{eq: l2 estimate of f for contraction} and the Lipschitz continuity of $g$, we deduce 
		\begin{equation}
			%\label{eq: final estimate for contraction}
			\begin{split}
				&\quad \, \left\|\partial_t u\right\|_{L^{\infty}(0,T_0;L^2(\Omega))}^2+\|u\|_{L^{\infty}(0,T_0;\widetilde{H}^s(\Omega))}^2\\
				&\leq CT_0^2\LC \|a\|^2_{L^p(\Omega)}+\left\|v^1\right\|^{2r}_{L^{\infty}(0,T_0;\widetilde{H}^s(\Omega))}+\left\|v^2\right\|^{2r}_{L^{\infty}(0,T_0;\widetilde{H}^s(\Omega))}\RC \\
				&\qquad  \qquad \cdot  \left\|v^1-v^2 \right\|^2_{L^{\infty}(0,T_0;\widetilde{H}^s(\Omega))}\\
				&\quad \, +CT_0^2\left\|\partial_tv^1-\partial_t v^2\right\|^2_{L^{\infty}(0,T_0;L^2(\Omega))}.
			\end{split}
		\end{equation}
		This implies
		\begin{equation}
			\label{eq: final estimate for contraction}
			\begin{split}
				&\left\|S(v^1)-S(v^2)\right\|_{T_0}\leq C_2T_0\LC 1+R^r\RC \left\|v^1-v^2 \right\|_{T_0}
			\end{split}
		\end{equation}
		as $u^1,u^2 \in X_{T_0,R}$. Note that $C_2>0$ only depends on the constants in Assumption~\ref{main assumptions on nonlinearities}, the parameters $n,s,r,p$, $\Omega$ and $a$. Thus, choosing
		\begin{equation}
			\label{eq: smaller T0}
			T_0\leq \min\left(\frac{R^2}{4C_0},\frac{R}{4C_1(1+R+R^{r+1})},\frac{1}{2C_2(1+R^r)}\right)
		\end{equation}
		we see that 
		\[
		\|S(v^1)-S(v^2)\|_{T_0}\leq \frac{1}{2} \left\|v^1-v^2 \right\|_{T_0},
		\]
		making $S\colon X_{T_0,R}\to X_{T_0,R}$  a (strict) contraction.\\
		
		\noindent{\textit{Step 3.}} Next we show that there exists a unique local solution to problem \eqref{eq: well-posedness nonlinear case}.
		
		\medskip
		
		Now, if we fix $R_0$ as in \eqref{eq: choosing parameters}, take $R\geq R_0$ and let $T_0$ satisfy \eqref{eq: smaller T0}, then we know that $S$ is a strict contraction from the Banach space $X_{T_0,R}$ to itself. Thus, by the Banach contraction principle there is a unique fixed point $u\in X_{T_0,R}$, that is a solution to \eqref{eq: well-posedness nonlinear case} with $T=T_0$. Next, we assert that if $\widetilde{u}\in C([0,T_0];\widetilde{H}^s(\Omega))\cap C^1([0,T_0];L^2(\Omega))$ is any other solution to \eqref{eq: well-posedness nonlinear case} on $[0,T_0]$, then we have $\widetilde{u}=u$ on $[0,T_0]$. First note that since $\widetilde{u}(0)=u_0,\partial_t\widetilde{u}(0)=u_1$, we can pick $0<T_1\leq T_0$ such that $\|\widetilde{u}\|_{T_1}\leq R_0$. Next fix any $R\geq R_0$. Then by Steps 1 and 2, we can find a $0<T_2\leq T_1$ such that $S\colon X_{T_2,R}\to X_{T_2,R}$ is a contraction. But then the contraction principle implies that $\widetilde{u}=u$ on $[0,T_2]$. We next claim that we have $T_2=T_0$. For this purpose, let us define
		\[
		I=\left\{t\in [0,T_0] \,:\, \widetilde{u}(t)=u(t)\right\}.
		\]
		We already know that $I$ is non-empty and closed, which follows from the continuity of the functions $\widetilde{u}$ and $u$. Next, we show that it is also open. Assume that $t_0\in I$ and set 
		\[
		A'=\max\LC \|u(\cdot, t_0)\|_{\widetilde{H}^s(\Omega)}, \|\partial_t u(\cdot, t_0)\|_{L^2(\Omega)}\RC.
		\]
		Let $R'>0$ satisfy $R'\geq R'_0=4A'$ and $R'>R$. Then Step 1, 2 imply that there exists $t_0<T'_0\leq T_0$ such that there is a unique fixed point of $S_{t_0}\colon X^{t_0}_{T'_0,R'}\to X^{t_0}_{T'_0,R'}$. Here the sub/subscripts indicate that we are considering everything over the time interval $[t_0,t_0+T'_0]$. Making $T'_0$ possibly smaller, we have $\widetilde{u},u\in X^{t_0}_{T'_0,R'}$ and hence it follows that $u=\widetilde{u}$ on $\left[t_0,t_0+T'_0\right]$. Using the time-reversal symmetry of the wave operator, the same holds on some interval of the form $\left[t_0-T''_0,t_0 \right]$. This shows that there is an open neighborhood of $t_0$ contained $I$. Therefore, $I$ is also open and by connectedness we have $I=[0,T_0]$. Hence, the constructed local solution is unique.\\
		
		\noindent{\textit{Step 4.}} Finally, we show that the constructed local solution can be extended to the unique global solution and the energy identity \eqref{eq: energy identity nonlinear case} holds.
		
		\medskip
		
		First note that by Step 3, we can extend the local solution to a map $u\colon \Omega\times [0,T_*)\to \R$, where $T_*\leq T$ denotes the maximal time of existence having the property that for any $T'<T_*$ the function $u\colon \Omega\times [0,T']\to\R$ solves \eqref{eq: well-posedness nonlinear case}.
		
		Now, let us fix any $0<T'<T_*$. Then we have $u\in H^1(0,T';L^2(\Omega))$ 
		and Fubini's theorem ensures $u(x,\cdot)\in H^1([0,T'])$ for a.e. $x\in \Omega$. Next, let us fix $x\in\Omega$ such that $u(x,\cdot)\in H^1([0,T'])$. By the embedding $H^1([0,T'])\hookrightarrow L^{\infty}([0,T'])$ and $F(x,\cdot)\in C^{0,1}_{loc}(\R)$ (see \eqref{eq:estimate on f}), we can modify $\tau\mapsto F(x,\tau)$ outside a neighborhood of the compact set $\left[-\|u(x,\cdot)\|_{L^{\infty}([0,T'])}, \, \|u(x,\cdot)\|_{L^{\infty}([0,T'])}\right]$ such that $F(x,\cdot)\in C^{0,1}(\R)$. Then \cite[Theorem~2.1.11]{Ziemer} guarantees $F(x,u(x,\cdot))\in H^1((0,T'))$ and there holds
		\[
		\partial_t F(x,u(x,t))=f(x,u(x,t))\partial_t u(x,t).
		\]
		Thus, the fundamental theorem of calculus shows
		\[
		F(x,u(x,t))=F(x,u_0(x))+\int_{0}^t f(x,u(x,s))\partial_t u(x,s)\,ds
		\]
		for all $0\leq t\leq T'$. Integrating this identity over $\Omega$ and using Fubini's theorem, we get
		\begin{equation}
			\label{eq: integral for F}
			\begin{split}
				\int_{\Omega}F(x,u(x,t)) \, dx&=\int_{\Omega}F(x,u_0(x))\, dx+\int_{0}^t \int_{\Omega}f(x,u(x,s))\partial_t u(x,s)\,dx ds
			\end{split}
		\end{equation}
		for all $0\leq t\leq T'$. As $u\colon \Omega\times [0,T']\to\R$ solves the linear problem
		\begin{equation}
			\begin{cases}
				\LC \partial_t^2 +(-\Delta)^s\RC u=h-f(\cdot,u)-g(\cdot,\partial_t u)&\text{in }\Omega_{T'},\\
				u=0&\text{in }(\Omega_e)_{T'},\\
				u(0)=u_0, \quad \partial_t u(0)=u_1&\text{in }\Omega.
			\end{cases}
		\end{equation}  
		By assumption and Step 1, we know that the right-hand side is in $L^2(\Omega_{T'})$ and we deduce from Theorem~\ref{thm: well-posedness linear} that $u$ satisfies the energy identity
		\[
		\begin{split}
			&  \quad \,  \left\|\partial_t u(t) \right\|_{L^2(\Omega)}^2+ \|u(t)\|_{\widetilde{H}^s(\Omega)}^2\\
			&= \left\|u_1 \right\|_{L^2(\Omega)}^2+\|u_0\|_{\widetilde{H}^s(\Omega)}^2+2 \left\langle h-f(\cdot,u)-g(\cdot,\partial_t u),\partial_t u \right\rangle_{L^2(\Omega_t)}
		\end{split}
		\]
		for all $0\leq t\leq T'$. By \eqref{eq: integral for F} this can be rewritten as
		\begin{equation}
			\label{eq: prelim energy id}
			\begin{split}
				&\quad \, \left\|\partial_t u(t) \right\|_{L^2(\Omega)}^2+\|u(t)\|_{\widetilde{H}^s(\Omega)}^2+2\int_{\Omega}F(x,u(x,t))\,dx+2\left\langle g(x,\partial_tu),\partial_t u \right\rangle_{L^2(\Omega_t)}\\
				&= \left\|u_1 \right\|_{L^2(\Omega)}^2+\|u_0\|_{\widetilde{H}^s(\Omega)}^2+2 \left\langle h,\partial_t u \right\rangle_{L^2(\Omega_t)}+2\int_{\Omega}F(x,u_0(x))\,dx
			\end{split}
		\end{equation}
		for all $0\leq t\leq T'$. Next, recall that $\inf_{\tau\in\R}F(x,\tau)\geq -C$ and therefore we obtain by H\"older's inequality the estimate
		\[
		\begin{split}
			&\quad \, \left\|\partial_t u(t) \right\|_{L^2(\Omega)}^2+\|u(t)\|_{\widetilde{H}^s(\Omega)}^2\\
			&\leq  2\left\langle h,\partial_t u \right\rangle_{L^2(\Omega_t)}-2 \left\langle g(x,\partial_tu),\partial_t u \right\rangle_{L^2(\Omega_t)}+C\\
			&\leq C'\int_0^t  \left\|\partial_t u(s) \right\|_{L^2(\Omega)}^2\,ds+C
		\end{split}
		\]
		for all $0\leq t\leq T'$, where $C,C'>0$ are independent of $T'$. Now, Gr\"onwall's inequality implies
		\begin{equation}
			\left\|\partial_t u(t) \right\|_{L^2(\Omega)}^2+\|u(t)\|_{\widetilde{H}^s(\Omega)}^2\leq C(1+C'te^{C't})
		\end{equation}
		for all $0\leq t\leq T'$. This shows that there exists $C>0$ such that
		\begin{equation}
			\label{eqL boundedness}
			\left\|\partial_t u(t) \right\|_{L^2(\Omega)}^2+\|u(t)\|_{\widetilde{H}^s(\Omega)}^2\leq C
		\end{equation}
		for all $0\leq t<T_*$. Next recall that the extension time $T_0$ only depends on the size of the initial condition and therefore, choosing $T'$ sufficiently close to $T_*$ we can extend the solution beyond $T_*$. For this recall the local uniqueness of solutions. This contradicts the claim that $T_*<T$ is maximal. Therefore we can conclude that $T_*=T$. Now, since $u\in C([0,T];\widetilde{H}^s(\Omega))\cap C^1([0,T];L^2(\Omega))$ we can pass to the limit $T'\to T$ in \eqref{eq: prelim energy id} and obtain \eqref{eq: energy identity nonlinear case}.
	\end{proof}
	
	\begin{proposition}[Well-posedness of nonlinear nonlocal wave equation with nonzero exterior condition]
		\label{prop: nonzero exterior condition}
		Let $\Omega\subset\R^n$ be a bounded Lipschitz domain, $T>0$ and $s>0$ a non-integer. Suppose that $f$ and $g$ satisfy Assumption~\ref{main assumptions on nonlinearities}. Then for any pair $(u_0,u_1)\in \widetilde{H}^s(\Omega)\times L^2(\Omega)$, $\varphi\in C^2([0,T];\widetilde{H}^{2s}(\Omega_e))$ and $h\in L^2(\Omega_T)$, there exists a unique solution $u$ of 
		\begin{equation}
			\label{eq: well-posedness linear case nonzero}
			\begin{cases}
				\partial_t^2u +(-\Delta)^su+f(x,u)+g(x,\partial_t u)=h    &\text{in }\Omega_T,\\
				u=\varphi     &\text{in }(\Omega_e)_T,\\
				u(0)=u_0,\quad \partial_t u(0)=u_1 &\text{in }\Omega.
			\end{cases}
		\end{equation}
		This means
		\begin{enumerate}[(i)]
			\item $u\in  C([0,T];H^s(\R^n))$  with $\partial_t u\in C([0,T];L^2(\R^n))$,
			\item $u=\varphi$ a.e. in $(\Omega_e)_T$,
			\item \eqref{eq: weak solution nonlin} holds for all $v\in\widetilde{H}^s(\Omega)$ in the sense of $\distr((0,T))$
			\item and $u(0)=u_0$, $\partial_tu(0)=u_1$.
		\end{enumerate}
	\end{proposition}
	
	\begin{proof}
		If $u\in C([0,T];H^s(\R^n))\cap C^1([0,T];L^2(\R^n))$ and $u=\varphi$ a.e. in $(\Omega_e)_T$, then $v=u-\varphi\in C([0,T];\widetilde{H}^s(\Omega))\cap C^1([0,T];L^2(\Omega))$. Hence, $u$ solves \eqref{eq: well-posedness linear case nonzero} if and only if $v$ solves \eqref{eq: well-posedness nonlinear case} with right-hand side given by $ h-(-\Delta)^s \varphi$. The latter problem is according to Theorem~\ref{thm: well-posedness nonlinear} well-posed and hence we obtain well-posedness for the problem \eqref{eq: well-posedness linear case nonzero}. 
	\end{proof}
	
	\begin{remark}
		\label{Rem: energy estimate}
		Let us adopt the assumptions of Proposition~\ref{prop: nonzero exterior condition}. Assume additionally $g=0$ and
		$\inf_{\tau\in\R} F(x,\tau)\geq 0$. Then the solution $u$ of \eqref{eq: well-posedness linear case nonzero} satisfies the energy estimate
		\begin{equation}\label{eq:nonlinear energy inequality}
			\begin{split}
				&\quad \, \Vert u-\varphi\Vert_{L^\infty(0,T;\wt H^s(\Omega))}^2 + 
				\Vert \p_t (u-\varphi)\Vert_{L^\infty(0,T;L^2(\Omega))}^2 \\
				&\leq C
				\left(
				\Vert u_0 \Vert_{\wt H^s(\Omega)}^2 + \Vert u_1 \Vert_{L^2(\Omega)}^2
				+ \Vert h-(-\Delta)^s\varphi \Vert_{L^2(\Omega_T)}^2 + \int_{\Omega}F(x,u_0(x))\,dx
				\right).
			\end{split}
		\end{equation}
		Particularly, when the initial data vanish the energy of the solution is bounded by the source and exterior data.
	\end{remark}
	\begin{proof}[Proof of Remark~\ref{Rem: energy estimate}]
		Let $u$ satisfy the equation
		\[
		\begin{cases}
			\partial_t^2u +(-\Delta)^su+f(x,u)=h    &\text{in }\Omega_T,\\
			u=\varphi     &\text{in }(\Omega_e)_T,\\
			u(0)=u_0,\quad \partial_t u(0)=u_1 &\text{in }\Omega.
		\end{cases}
		\]
		Then $v:=u-\varphi$ satisfies
		\[
		\begin{cases}
			\partial_t^2v +(-\Delta)^sv+f(x,v)=h-(-\Delta)^s\varphi    &\text{in }\Omega_T,\\
			v=0    &\text{in }(\Omega_e)_T,\\
			u(0)=u_0,\quad \partial_t u(0)=u_1 &\text{in }\Omega.
		\end{cases}
		\]
		The energy equality of Theorem~\ref{thm: well-posedness nonlinear} yields
		\begin{equation}
			\begin{split}
					&\quad\,  \left\|\partial_t v(t) \right\|_{L^2(\Omega)}^2+\|v(t)\|_{\widetilde{H}^s(\Omega)}^2+2\int_{\Omega}F(x,v(x,t))\,dx\\
				&= \left\|u_1 \right\|_{L^2(\Omega)}^2+\|u_0\|_{\widetilde{H}^s(\Omega)}^2+2 \left\langle h-(-\Delta)^s\varphi,\partial_t v \right\rangle_{L^2(\Omega_t)}+2\int_{\Omega}F(x,u_0(x))\,dx.    
			\end{split}
		\end{equation}
		Assuming that $\inf_{\tau\in\R} F(x,\tau)\geq 0$ we can estimate by H\"older's inequality
		\begin{align*}
			&\quad \, \left\|\partial_t v(t) \right\|_{L^2(\Omega)}^2+\|v(t)\|_{\widetilde{H}^s(\Omega)}^2 \\
			&\leq \left\|u_1 \right\|_{L^2(\Omega)}^2+\|u_0\|_{\widetilde{H}^s(\Omega)}^2\\
			&\quad+2 \left\langle h-(-\Delta)^s\varphi,\partial_t v \right\rangle_{L^2(\Omega_t)}
			+2\int_{\Omega}F(x,u_0(x))\,dx\\
			&\leq \left\|u_1 \right\|_{L^2(\Omega)}^2+\|u_0\|_{\widetilde{H}^s(\Omega)}^2\\
			&\quad + 2\left\| h-(-\Delta)^s\varphi \right\|_{L^2(\Omega_T)}  \Vert \p_t v(t) \Vert_{L^2(\Omega_T)} 
			+2\int_{\Omega}F(x,u_0(x))\,dx.
		\end{align*}
		Note that since $\varphi \in C^2([0,T];\wt H^{2s}(\Omega_e))$ we have $v|_{\Omega_T}=(u-\varphi)|_{\Omega_T}=u|_{\Omega_T}$.
		As the right-hand side is independent of $t$ we may take supremum over $t\in [0,T]$.
		Using the elementary inequality $2ab \leq \frac{1}{\eps}a^2 + \eps b^2$ and $\Vert \p_tv(t)\Vert_{L^2(\Omega_T)}^2 \leq C\Vert\p_t v(t) \|_{L^{\infty}(0,T; L^2(\Omega))}^2$ we may absorb the norm $\Vert\p_t v(t) \|_{L^{\infty}(0,T; L^2(\Omega))}^2$ on the left-hand side at the cost of a larger constant and obtain the claim.
	\end{proof}
	
	\subsection{The DN map}\label{sec: DN}
	
	With the well-posedness of the nonlinear nonlocal wave equation at hand, we can rigorously define the DN map. 
	
	\begin{definition}[The DN map]\label{def: the DN map}
		Let $\Omega \subset \R^n$ be a bounded domain, $T>0$ and $s>0$ a non-integer. Suppose the nonlinearities $f,$ $g$ satisfy the conditions in Assumption \ref{main assumptions on nonlinearities} and $(u_{0},u_1)\in \wt H^s(\Omega)\times L^2(\Omega)$. Then we define the DN map $\Lambda^{f,g}_{u_{0},u_{1}}$ related to
		\eqref{eq: nonlinear wave equation}
		by
		\begin{equation}
			\label{eq: modified DN map}
			\begin{split}
				\left\langle \Lambda^{f,g}_{u_0,u_1}\varphi,\psi \right\rangle \vcentcolon =	\int_{\R^n_T}(-\Delta)^{s/2}u\, (-\Delta)^{s/2}\psi\,dxdt,
			\end{split}
		\end{equation}
		for all $\varphi,\psi\in C_c^{\infty}((\Omega_e)_T)$, where $u\in C([0,T];H^s(\R^n))\cap C^1([0,T];L^2(\Omega))$ is the unique solution of
		\[
		\begin{cases}
			\partial_t^2u +(-\Delta)^su + f(x,u)+g(x,\partial_tu)=0 &\text{ in }\Omega_T,\\
			u=\varphi  &\text{ in }(\Omega_e)_T,\\
			u(x,0)=u_0(x), \quad \partial_t u(x,0)=u_1(x) &\text{ in }\Omega
		\end{cases}
		\]
		(see Proposition~\ref{prop: nonzero exterior condition}). 
	\end{definition}

	\section{Inverse problems for the nonlinear nonlocal wave equation}
	\label{sec: inverse problem}
	
	In this section we move on the inverse problem for the (nonlinear) nonlocal wave equation. First, in Section~\ref{sec: Runge} we establish the Runge approximation property for the linear nonlocal wave equation with nonnegative potentials, and arbitrary initial data. Then, in Section~\ref{sec: proofs inverse problems} we prove the main results of this work (Theorem~\ref{Thm: recovery of nonlinearity} and \ref{Thm: recovery of initial values}). 
	
	\subsection{Runge approximation}
	\label{sec: Runge}
	
	As usual we will deduce the Runge approximation property from the Hahn--Banach theorem and the UCP of the fractional Laplacian.
	\begin{proposition}[Runge approximation]\label{prop: runge}
		Let $\Omega\subset \R^n$ be a bounded Lipschitz domain, $W\subset\Omega_e$ an arbitrary open set, $s>0$ a non-integer and $T>0$. Suppose that $q\in L^p(\Omega)$ is nonnegative\footnote{This assumption is included for simplicity and the result remains true, for example, if one assumes instead $q\in L^{\infty}(\Omega)$.}, where $p$ is given by \eqref{eq: cond on p},
		and $(u_0,u_1)\in \wt H^s(\Omega)\times L^2(\Omega)$ are fixed initial conditions.
		Consider the \emph{Runge set} 
		\begin{align*}
			\mathcal{R}_W^{u_0,u_1}:=\left\{ u_{\varphi}\big|_{\Omega_T} \,:\, \varphi\in C^\infty_c (W_T) \right\},
		\end{align*}
		where $u_\varphi\in C([0,T];H^s(\R^n))\cap C^1([0,T]; L^2(\R^n))$ is the unique solution to 
		\begin{equation}\label{eq:runge with zero initial}
			\begin{cases}
				\partial_t^2u +(-\Delta)^su + qu = 0 &\text{ in }\Omega_T,\\
				u=\varphi  &\text{ in }(\Omega_e)_T,\\
				u(0)=u_0, \quad \partial_t u(0)=u_1 &\text{ in }\Omega.
			\end{cases}
		\end{equation}
		Then $\mathcal{R}_W^{u_0,u_1}$ is dense in $L^2(\Omega_T)$.
	\end{proposition}
	\begin{proof}
		
		Let us first observe that without loss of generality it suffices to show that $\mathcal{R}_W:=\mathcal{R}_W^{0,0}$ is dense in $L^2(\Omega_T)$. Indeed, if $g\in L^2(\Omega_T)$ and $u_0^{u_0,u_1}$ denotes the unique solution of \eqref{eq:runge with zero initial} with $\varphi=0$, then the function $\wt g:= g - u_0^{u_0,u_1}$ belongs to $L^2(\Omega_T)$. Now, assuming $\mathcal{R}_W$ is dense in $L^2(\Omega_T)$, we can find a sequence $(v_k|_{\Omega_T})_{k=1}^\infty\subset \mathcal{R}_W$ such that $v_k|_{\Omega_T}\to \wt g|_{\Omega_T}$ in $L^2(\Omega_T)$ as $k\to\infty$. Defining $w_k:=v_k+u_0^{u_0,u_1}$, we immediately see that $w_k|_{\Omega_T} \in \mathcal{R}_W^{u_0,u_1}$ and 
		\[
		\left. w_k\right|_{\Omega_T} = \left.\LC v_k+u_0^{u_0,u_1}\RC \right|_{\Omega_T}\to  \left. \LC \wt g + u_0^{u_0,u_1}\RC \right|_{\Omega_T} = g\quad\text{as}\quad  k\to\infty
		\]
		in $L^2(\Omega_T)$.
		
		Thus, we may now suppose that $u_0=u_1=0$. We combine the strategy of \cite[Theorem~3.1]{KLW2022} and \cite[Theorem~4.3]{RZ-unbounded}. Since 
		$\mathcal{R}_W \subset L^2(\Omega_T)$
		is a subspace it is enough by the Hahn--Banach theorem to show that if 
		$F\in L^2(\Omega_T)$
		vanishes on $\mathcal{R}_W$, then $F=0$. Hence, choose any 
		$F\in L^2(\Omega_T)$
		and assume that 
		\begin{equation}
			\left\langle F,u_{\varphi}-\varphi \right\rangle_{L^2(\Omega_T)}=0\quad\text{for all}\quad \varphi\in C_c^{\infty}(W_T).
		\end{equation}
		Next, let 
		\[
		w_F\in C([0,T];\widetilde{H}^s(\Omega))\text{ with }\partial_t w_F \in C([0,T];L^2(\Omega))\text{ and }\partial_t^2w_F\in L^2(0,T;H^{-s}(\Omega))
		\]
		be the unique solution to the adjoint equation 
		\begin{equation}
			\begin{cases}
				\partial_t^2w +(-\Delta)^sw+qw = F&\text{ in }\Omega_T,\\
				w=0  &\text{ in }(\Omega_e)_T,\\
				w(T)= \partial_t w(T)=0 &\text{ in }\Omega,
			\end{cases}
		\end{equation}
		which can be obtained by \cite[Chapter~3, Theorem~8.1-8.2]{LionsMagenesVol1} and a subsequent time reversal (i.e., $t\mapsto T-t$). Next, we show the following assertion.
		\begin{claim}
			\label{claim: integration by parts}
			There holds
			\begin{equation}
				\label{eq: integration by parts}
				\int_0^T \left\langle \partial_t^2(u_{\varphi}-\varphi),w_F\right\rangle \,dt=\int_0^T \langle \partial_t^2w_F,(u_{\varphi}-\varphi)\rangle \,dt.
			\end{equation}
		\end{claim}
		\begin{proof}[Proof of Claim~\ref{claim: integration by parts}]
			Let us set $v_{\varphi}=u_{\varphi}-\varphi$ and consider for each $\eps>0$ the following regularized problem
			\begin{equation}
				\label{eq: regularization for equation of v}
				\begin{cases}
					\partial_t^2v +\eps((-\Delta)^s+q)\partial_t v+(-\Delta)^sv + qv = -(-\Delta)^s\varphi &\text{ in }\Omega_T,\\
					v=0  &\text{ in }(\Omega_e)_T,\\
					v(0)= \partial_t v(0)=0 &\text{ in }\Omega
				\end{cases}
			\end{equation}
			and
			\begin{equation}
				\label{eq: regularization for equation of w}
				\begin{cases}
					\partial_t^2w -\eps ((-\Delta)^s+q) \partial_t w+(-\Delta)^sw+qw = F&\text{ in }\Omega_T,\\
					w=0  &\text{ in }(\Omega_e)_T,\\
					w(T)= \partial_t w(T)=0 &\text{ in }\Omega.
				\end{cases}
			\end{equation}
			It is well-known that by \cite[Chapter 3, Theorem~8.3]{LionsMagenesVol1}, the above problems have unique solutions 
			\begin{equation}
				\begin{split}
					v_{\varphi}^{\eps}&\in C([0,T];\widetilde{H}^s(\Omega))\text{ with } \begin{cases}
						&\hspace{-0.3cm}\partial_t v_{\varphi}^{\eps}\in L^2(0,T;\widetilde{H}^s(\Omega))\cap C([0,T];L^2(\Omega))\\
						&\hspace{-0.3cm}\partial_t^2 v_{\varphi}^{\eps}\in L^2(0,T;H^{-s}(\Omega))
					\end{cases} \\
					w_F^{\eps}&\in C([0,T];\widetilde{H}^s(\Omega))\text{ with } \begin{cases}
						&\hspace{-0.3cm}\partial_t w_F^{\eps}\in L^2(0,T;\widetilde{H}^s(\Omega))\cap C([0,T];L^2(\Omega))\\
						&\hspace{-0.3cm}\partial_t^2 w_F^{\eps}\in L^2(0,T;H^{-s}(\Omega))
					\end{cases}
				\end{split}
			\end{equation}
			(see also \cite{zimmermann2024calderon}) and as $\eps\to 0$, one has
			\begin{equation}
				\label{eq: convergence}
				\begin{split}
					v_{\varphi}^{\eps}\to v_\varphi &\text{ in }C([0,T];\widetilde{H}^s(\Omega)),\\
					\partial_t v_{\varphi}^{\eps}\to \partial_t v_\varphi &\text{ in }C([0,T];L^2(\Omega)),\\
					\partial_t^2v_{\varphi}^{\eps}\weak \partial_t^2 v_\varphi &\text{ in } L^2(0,T;H^{-s}(\Omega)).
				\end{split}
			\end{equation}
			Note that here we use the assumption $q\geq 0$ and otherwise we would have to include a term $\eps\lambda$ for some $\lambda\in\R$ in the regularized problems \eqref{eq: regularization for equation of v} and \eqref{eq: regularization for equation of w} such that $(-\Delta)^s+q+\lambda$ is coercive.
			The last convergence in \eqref{eq: convergence} follows by \cite[Chapter 3, eq.~(8.74)]{LionsMagenesVol1} and in particular it implies that
			\begin{equation}
				\label{eq: convergence second time derivative}\partial_t^2v_{\varphi}^{\eps}\weakstar \partial_t^2 v_\varphi \text{ in } L^2(0,T;H^{-s}(\Omega))
			\end{equation}
			The convergence results in \eqref{eq: convergence} clearly hold for the functions $w_F^\eps$ and $w_F$, too. Now, using an integration by parts, we can easily get
			\[
			\int_0^T \left\langle \partial_t^2 v_{\varphi}^{\eps},w_F^\eps \right\rangle \,dt=\int_0^T \left\langle \partial_t^2w_F^\eps,v_{\varphi}^\eps \right\rangle \,dt
			\]
			for any $\eps>0$ (see \cite[eq. (4.1)]{zimmermann2024calderon}). The convergence in \eqref{eq: convergence} and \eqref{eq: convergence second time derivative} allow us to pass to the limit $\eps\to 0$, which yields
			\[
			\int_0^T \left\langle \partial_t^2 v_{\varphi},w_F \right\rangle \,dt=\int_0^T \left\langle \partial_t^2w_F,v_{\varphi} \right\rangle \,dt.
			\]
			Recalling that $v_\varphi=u_\varphi-\varphi$, we can conclude the proof.
		\end{proof}

		Then we can test the equation for $u_{\varphi}-\varphi$ against $w_F$ and correspondingly test the equation for $w_F$ by $u_{\varphi}-\varphi$. By Claim~\ref{claim: integration by parts} we may compute
		\[
		\begin{split}
			0&=\left\langle F,u_{\varphi}-\varphi\right\rangle_{L^2(\Omega_T)} \\
			&=\left\langle (\partial_t^2 +(-\Delta)^s +q)w_F,u_{\varphi}-\varphi\right\rangle\\
			&= \left\langle (\partial_t^2+(-\Delta)^s +q)(u_{\varphi}-\varphi),w_F\right\rangle\\
			&=\left\langle (-\Delta)^s \varphi,w_F \right\rangle,
		\end{split}
		\]
		for all $\varphi\in C_c^{\infty}(W_T)$. This implies that $w_F$ satisfies
		\[
		(-\Delta)^s w_F(x,t)=w_F(x,t)=0\quad\text{for}\quad  x\in W,
		\]
		for a.e. $0<t<T$. By the UCP of the fractional Laplacian \cite[Theorem~1.2]{GSU20}, this gives $w_F=0$ in $\R^n_T$. This in turn implies $F=0$ as we wish.
	\end{proof}

	\subsection{Proofs related to the inverse problems}
	
	To study the inverse problem of recovering the nonlinear terms $f_j$ we need certain mapping properties of the Nemytskii operators.
 
	\begin{lemma}[Continuity of Nemytskii operators, {\cite[Lemma~3.6]{zimmermann2024calderon}}]
		\label{lemma: Continuity of Nemytskii operators}
		Let $\Omega\subset \R^n$ be a bounded open set, $T>0$, and $1\leq q,p<\infty$. Assume that $f\colon\Omega\times \R\to\R$ is a Carath\'eodory function satisfying
		\begin{equation}
			\label{eq: cond continuity nemytskii}
			|f(x,\tau)|\leq a+b|\tau|^{\alpha}
		\end{equation}
		for some constants $a,b\geq 0$ and $0<\alpha \leq \min(p,q)$. Then the Nemytskii operator $f$, defined by
		\begin{equation}
			\label{eq: Nemytskii operator}
			f(u)(x,t)\vcentcolon = f(x,u(x,t))
		\end{equation}
		for all measurable functions $u\colon \Omega_T\to\R$, maps continuously $L^q(0,T;L^p(\Omega))$ into $L^{q/\alpha}(0,T;L^{p/\alpha}(\Omega))$.
	\end{lemma}

	\label{sec: proofs inverse problems}
	\begin{proof}[Proof of Theorem \ref{Thm: recovery of initial values}]
		
		{\it Unique determination of initial data.}
		
		\medskip
		
		\noindent First note that the condition \eqref{same DN map in thm} is equivalent to 
		\begin{equation}
			(-\Delta)^s u^{(1)}_0=(-\Delta)^su^{(2)}_0 \text{ in }(W_2)_T,
		\end{equation}
		where $u^{(j)}_0$ is the unique solution to \eqref{eq: nonlinear wave equation j=1,2} for $j=1,2$ with the same exterior condition $\varphi=0$.
		Then the UCP for the fractional Laplacian (Proposition \ref{prop:UCP}) yields that $u_0^{(1)}=u_0^{(2)}$ in $\R^n_T$. This in turn implies 
		\begin{equation}\label{eq: unique of initial}
			u_0:=u_{0,1}=u_{0,2} \quad \text{and} \quad u_1:=u_{1,1}=u_{1,2} \quad \text{in}\quad \Omega.
		\end{equation}
		This concludes the proof.
	\end{proof}
	
	\begin{remark}
		Even if the exterior data $\varphi=0$,  the corresponding solution of \eqref{equation det of initial j=12} does not need to vanish, since the initial data may be nonzero.
	\end{remark}

	\begin{proof}[Proof of Corollary \ref{Cor: simul near}]
		
		Similarly as above, by the UCP it follows that the potential terms are equal:
		\begin{equation}\label{eq: f1 equals f2}
			a_1 u^{(1)}  = a_2 u^{(1)}  .
		\end{equation}
		This implies that there holds
		\[
		\int_{\Omega_T}a_1u\, dxdt  = \int_{\Omega_T} a_2u \, dxdt
		\]
		for all solutions $u$ of a linear wave equation with potential $a_j$. Since we assume $0\leq a_j\in L^p(\Omega)$ with \eqref{eq: cond on p} (implying $p\geq 2$) or $a_j\in L^{\infty}(\Omega)$, it follows from Proposition~\ref{prop: runge} that the solutions $u$ are dense in $L^2(\Omega_T)$. Finally, using H\"older's inequality we get 
		\[
		\int_{\Omega_T}(a_1-a_2)\psi\,dxdt=0
		\]
		for all $\psi\in C_c^{\infty}(\Omega_T)$ and hence $a_1=a_2$.
	\end{proof}
	
	Last but not least, let us prove the recovery of the nonlinear term.
	
	\begin{proof}[Proof of Theorem \ref{Thm: recovery of nonlinearity}]
		\noindent {\it Unique determination of the nonlinear term $f(x,u)$.}
		
		\medskip
		\noindent Let $\eps>0$. We start by observing that the solution $u_\eps$ of
		\[
		\begin{cases}
			\partial_t^2u +(-\Delta)^su+f(x,u)=0  &\text{in }\Omega_T,\\
			u=\eps\varphi  &\text{in }(\Omega_e)_T,\\
			u(0)= \partial_t u(0)=0 &\text{in }\Omega,
		\end{cases}
		\]
		can be expanded as $u_\eps = \eps v+ R_\eps$, where
		\begin{equation}\label{eq:ip linear wave}
			\begin{cases}
				\partial_t^2v +(-\Delta)^sv=0  &\text{in }\Omega_T,\\
				v=\varphi   &\text{in }(\Omega_e)_T,\\
				v(0)= \partial_t v(0)=0 &\text{in }\Omega
			\end{cases}    
		\end{equation}
		and 
		\begin{equation}\label{eq:remainder R}
			\begin{cases}
				\partial_t^2R_\eps +(-\Delta)^s R_\eps=  -f(x,u_\eps)  &\text{in }\Omega_T,\\
				R_\eps=0   &\text{in }(\Omega_e)_T,\\
				R_\eps(0)= \partial_t R_\eps(0)=0 &\text{in }\Omega.
			\end{cases}    
		\end{equation}
		Indeed, the function $R_\eps = u_\eps -\eps v$ satisfies \eqref{eq:remainder R}. 
		Note that by \eqref{eq: estimate Hs} for $u\in C([0,T];H^s(\R^n))$ and $f(x,0)=0$ we have the estimate
		\[
		\left\| f_j(\cdot,u)\right\|_{L^2(\Omega_T)} \lesssim \Vert u \Vert_{L^\infty(0,T;H^s(\R^n))}^{r+1}.
		\]
		Since $u_\eps\in C([0,T];H^s(\R^n))$ and $f(u_\eps)\in L^2(\Omega_T)$, then $R_\eps\in C([0,T];\wt H^s(\Omega))\cap C^1([0,T];L^2(\Omega))$ and the energy identity \eqref{eq: energy identity linear case} implies:
		\begin{equation}\label{eq: estimate for R wrt u}
			\begin{split}
				\Vert R_\eps \Vert_{L^\infty(0,T;\wt H^s(\Omega))} + \Vert \p_t R_\eps
				\Vert_{L^\infty(0,T; L^2(\Omega))}&
				\lesssim  \Vert f(u_\eps)\Vert_{L^2(\Omega_T)} \\
				&\lesssim  \Vert u_\eps \Vert_{L^\infty(0,T;H^s(\R^n))}^{r+1}.   
			\end{split}
		\end{equation}
		By Remark~\ref{Rem: energy estimate}, we get
		\begin{equation}\label{eq: estimate for u wrt phi}
			\begin{split}
				&\quad \, \Vert u_\eps \Vert_{L^{\infty}(0,T; H^s(\R^n))} +
				\Vert \p_t u_\eps \Vert_{L^\infty(0,T; L^2(\R^n))} \\
				&\lesssim 
				\Vert u_\eps - \eps\varphi \Vert_{L^\infty(0,T; \wt H^s(\Omega))} 
				+
				\Vert \p_t (u_\eps-\eps\varphi) \Vert_{L^\infty(0,T; L^2(\Omega))}\\
				&\quad \, +
				\Vert \eps\varphi \Vert_{L^\infty(0,T; \wt H^s(\Omega_e))} 
				+
				\Vert\eps \p_t \varphi \Vert_{L^\infty(0,T; L^2(\Omega_e))}\\
				&\lesssim \eps\left(
				\Vert (-\Delta)^{s} \varphi\Vert_{L^2(\Omega_T)}
				+
				\Vert \varphi \Vert_{L^\infty(0,T; \wt H^s(\Omega_e))} 
				+
				\Vert \p_t \varphi \Vert_{L^\infty(0,T; L^2(\Omega_e))}
				\right).
			\end{split}
		\end{equation}
		Therefore, by combining \eqref{eq: estimate for R wrt u} with \eqref{eq: estimate for u wrt phi} we obtain
		\begin{equation}\label{eq: estimate for R wrt eps}
			\begin{split}
				&\quad \, \Vert R_\eps \Vert_{L^\infty(0,T;\wt H^s(\Omega))} + \Vert \p_t R_\eps \Vert_{L^\infty(0,T; L^2(\Omega))}\\
				&\lesssim \eps^{r+1}\left(
				\Vert (-\Delta)^s \varphi\Vert_{L^2(\R^n_T)}
				+
				\Vert \varphi \Vert_{L^\infty(0,T; \wt H^s(\Omega))} 
				+
				\Vert \p_t \varphi \Vert_{L^\infty(0,T; L^2(\Omega_e))}
				\right)^{r+1}.    
			\end{split}
		\end{equation}
		From \eqref{eq: f1 equals f2} we know
		\[
		f_1(x,u_\eps)=f_2(x,u_\eps)
		\]
		which is equivalent to
		\[
		f_1(x,\eps v + R_\eps) = f_2(x,\eps v+R_\eps).
		\]
		By the $(r+1)$-homogeneity of $f_j$ we then find that
		\[
		f_1(x,v+\eps^{-1} R_\eps) = f_2(x,v+\eps^{-1}R_\eps).
		\]
		Next, we may use the estimate \eqref{eq: estimate for R wrt eps} for $R_\eps\in C([0,T]; \wt H^s(\Omega)) \subset L^2(\Omega_T)$ with the bound 
		\[
		\Vert R_\eps\Vert_{L^2(\Omega_T)}\lesssim \eps^{r+1}
		\]
		to deduce that 
		\[
		\eps^{-1}R_\eps \to 0\text{ in } L^2(\Omega_T) \text{ as }\eps\to 0.
		\]
		Therefore by continuity of the Nemytskii operator $f_j$ from $L^2(\Omega_T)$ to $L^\frac{2}{r+1}(\Omega_T)$ (see Lemma~\ref{lemma: Continuity of Nemytskii operators}) we get
		\[
		f_j(v+\eps^{-1}R_\eps)\to f_j(v)\text{ in } L^\frac{2}{r+1}(\Omega_T),\quad \text{as} \quad \eps\to 0,
		\]
	for all $0<r\leq 1$.
 
		Thus we have obtained that $f_1(x,v)=f_2(x,v)$ in $L^\frac{2}{r+1}(\Omega_T)$ for all solutions $v$ of the linear wave equation \eqref{eq:ip linear wave}. Finally, we can apply the Runge approximation property of solutions $v$ to linear wave equations with $q=0$. Let $(v_k)_{k=1}^\infty\subset \mathcal{R}_{W_1}$ be a sequence of solutions to the linear nonlocal wave equations with $v_k\big|_{\Omega_T}\to 1$ in $L^2(\Omega_T)$ given by Proposition~\ref{prop: runge}. Then continuity of $f_j: L^2(\Omega_T)\to L^\frac{2}{r+1}(\Omega_T)$, when $0<r\leq 1$, yields
		\[
		\left\| f_j(x,1)-f_j(x,v_k)\right\|_{L^\frac{2}{r+1}(\Omega_T)} \to 0\quad\text{as } k\to\infty.
		\]
		Thus we obtain 
		\[
		f_1(x,1)=f_2(x,1)
		\]
		for almost all $x\in\Omega$. Therefore, by homogeneity we recover $f_1(x,\tau)=f_2(x,\tau)$ for a.e. $x\in\Omega$ and $\tau\in \R$.
	\end{proof}
	
	\begin{remark}
		It is noteworthy that we do not need to use the (in local case) commonly used \emph{higher order linearization} (popularized for inverse problems in \cite{KLU18,LLLS21} for both hyperbolic and elliptic equations). In the local case, the higher order linearization is used to builds products $v_1\cdots v_k$ of solutions to linear equations. Then one can show that these products become dense in suitable function spaces. In the nonlocal case the UCP of the fractional Laplacian directly establishes that $f_1(u)=f_2(u)$ for any solution $u$ of the nonlinear nonlocal wave equation. Then it suffices to linearize once to recover $f_1(v)=f_2(v)$ for any solution $v$ of a linear wave equation. Finally, the Runge approximation property yields the equality $f_1=f_2$. In a sense, this is consequence of the strength of nonlocality and its UCP. This technique applies to many equations of similar type, including as a typical example the fractional Schr\"odinger equation $(-\Delta)^s u +q(u)=0$, since all one needs is the UCP and the Runge approximation.
	\end{remark}

	\medskip 
	
	\noindent\textbf{Acknowledgments.} 
	\begin{itemize}
		\item Y.-H.~Lin was partially supported by the National Science and Technology Council (NSTC) Taiwan, under the project 113-2628-M-A49-003. Y.-H. Lin is also a Humboldt research fellow. 
		\item T.~Tyni was supported by the Research Council of Finland (Flagship of Advanced Mathematics for Sensing, Imaging and Modelling grant 359186) and by the Emil Aaltonen Foundation. T.~Tyni is grateful for the hospitality provided by T.~Balehowsky during his visit to University of Calgary in 2023.
		\item P.~Zimmermann was supported by the Swiss National Science Foundation (SNSF), under the grant number 214500.
	\end{itemize}

	\bibliography{refs} 
	
	\bibliographystyle{alpha}
	
\end{document}